\documentclass[11pt]{amsart}
\usepackage[a4paper]{geometry}
\usepackage[colorlinks,bookmarksopen=true]{hyperref}
\usepackage[english]{babel}
\usepackage{amssymb}
\usepackage{latexsym}
\usepackage{mathrsfs}
\usepackage{xspace}
\usepackage{enumerate}
\usepackage{paralist}
\usepackage[all, cmtip]{xy}
\newtheorem{prop}{Proposition}[section]
\newtheorem{thm}[prop]{Theorem}
\newtheorem*{thm*}{Theorem}

\newtheorem*{addendum*}{Addendum}
\newtheorem{cor}[prop]{Corollary}
\newtheorem{lem}[prop]{Lemma}
\newtheorem{thmintro}{Theorem}

\newtheorem{corintro}[thmintro]{Corollary}
\newtheorem{propintro}[thmintro]{Proposition}

\newtheorem*{convention*}{Convention}
\theoremstyle{definition}
\newtheorem*{defn*}{Definition}
\newtheorem{defn}[prop]{Definition}
\newtheorem{remark}[prop]{Remark}
\newtheorem{remarks}[prop]{Remarks}

\newtheorem*{scholium*}{Scholium}
\theoremstyle{remark}
\newtheorem{example}[prop]{Example}
\newtheorem*{example*}{Example}
\numberwithin{equation}{section}

\newcommand{\vareps}{\varepsilon}
\newcommand{\fhi}{\varphi}
\newcommand{\ro}{\varrho}
\newcommand{\teta}{\vartheta}

\newcommand{\EE}{\mathbf{E}}
\newcommand{\FF}{\mathbf{F}}

\newcommand{\HH}{\mathbf{H}}

\newcommand{\NN}{\mathbf{N}}

\newcommand{\QQ}{\mathbf{Q}}
\newcommand{\RR}{\mathbf{R}}

\newcommand{\ZZ}{\mathbf{Z}}

\newcommand{\GL}{\mathrm{GL}}
\newcommand{\SL}{\mathrm{SL}}

\newcommand{\la}{\langle}
\newcommand{\ra}{\rangle}
\newcommand{\inv}{^{-1}}

\newcommand{\centra}{\mathscr{Z}}

\newcommand{\se}{\subseteq}

\newcommand{\lra}{\longrightarrow}

\newcommand{\Id}{\mathrm{Id}}
\def\bs#1.{
              \def\temp{#1}
              \ifx\temp\empty
                   \mathcal{B}
              \else
                   \mathcal{B}(#1)
              \fi
}
\newcommand{\cat}{{\upshape CAT(0)}\xspace}
\newcommand{\tangle}[2]
{\angle_\mathrm{T}(#1,#2)}
\newcommand{\aangle}[3]
{\angle_{#1}(#2,#3)}
\newcommand{\cangle}[3]
{\overline{\angle}_{#1}(#2,#3)}

\DeclareMathOperator{\Stab}{Stab}  
\DeclareMathOperator{\Ad}{Ad}\DeclareMathOperator{\Ker}{Ker}

\DeclareMathOperator{\Isom}{Is}

\DeclareMathOperator{\LF}{Rad_{\mathscr{L\!E\!}}}

\newcommand{\bd}{\partial} 

\def\Aut{\mathop{\mathrm{Aut}}\nolimits}
\def\Inn{\mathop{\mathrm{Inn}}\nolimits}

\def\Op{\mathop{\mathrm{Opp}}\nolimits}
\def\Ant{\mathop{\mathrm{Ant}}\nolimits}

\makeindex
\begin{document}
\title[Fixed points and amenablility]{Fixed points and amenability\\ in non-positive curvature}
\author[P.-E. Caprace]{Pierre-Emmanuel Caprace*}
\address{UCL -- Math, Chemin du Cyclotron 2, 1348 Louvain-la-Neuve, Belgium}
\email{pe.caprace@uclouvain.be}
\thanks{* F.R.S.-FNRS research associate. Supported in part by FNRS grant F.4520.11 and by the ERC}
\author[N. Monod]{Nicolas Monod$^\ddagger$}
\address{EPFL, 1015 Lausanne, Switzerland}
\email{nicolas.monod@epfl.ch}
\thanks{$^\ddagger$Supported in part by the Swiss National Science Foundation and the ERC}
\date{March 2012}
%
\begin{abstract}
Consider a proper cocompact \cat space $X$. We give a complete algebraic characterisation of amenable groups of isometries of $X$.
For amenable \emph{discrete} subgroups, an even narrower description is derived, implying $\QQ$-linearity in the torsion-free case.

We establish \emph{Levi decompositions} for stabilisers of points at infinity of $X$, generalising the case of linear algebraic groups to $\Isom(X)$. A geometric counterpart of this sheds light on the refined bordification of $X$ (\`a la Karpelevich) and leads to a converse to the Adams--Ballmann theorem. It is further deduced that \emph{unimodular} cocompact groups cannot fix any point at infinity except in the Euclidean factor; this fact is needed for the study of \cat lattices.

Various fixed point results are derived as illustrations.
\end{abstract}
\maketitle

\setcounter{tocdepth}{1}    
\tableofcontents

\section{Introduction}

\subsection{Amenable isometry groups}
A celebrated theorem by Tits asserts that an arbitrary group $G \se \GL(V)$ of linear transformations of a finite-dimensional vector space $V$ over any field is subjected to the following alternative:   \emph{either $G$ contains a non-abelian free subgroup, or $G$ is soluble-by-\{locally finite\}} (see Theorems~1 and~2 in~\cite{Tits72}). In particular a subgroup $G \se \GL(V)$ is amenable if and only if it is soluble-by-\{locally finite\}. 
The importance of that result stimulated  since then an active search for larger classes of groups satisfying a similar alternative. It is in particular a notorious open problem to obtain a version of the Tits alternative for groups $G \se \Isom(X)$, where $X$ is a cocompact proper \cat space, i.e. a non-positively curved proper metric space with a cocompact isometry group.

While considering extensions of the Tits alternative to wider families of groups, it is natural to split the problem into two sub-questions, namely: 

\begin{enumerate}[(1)]
\item Does every non-amenable subgroup  contain a free group? 

\item  What is the algebraic structure of amenable subgroups? 
\end{enumerate}

Our first goal in this paper is to provide a complete answer to Question (2) for isometry groups of a cocompact proper \cat space. To this end, we recall that the \textbf{locally elliptic radical} $\LF$ of a locally compact group is the largest normal subgroup which can be written as increasing union of compact groups (see~\cite{Platonov}). In case of discrete groups, \emph{locally elliptic} is thus a synonym of \emph{locally finite}. 

\medskip
The following theorem shows that a subgroup is amenable if and only if it has a specific canonical decomposition into pieces that are either connected soluble or discrete soluble or locally elliptic, in analogy with  Tits' description of amenable linear groups.

\begin{thmintro}\label{thm:StructureAmenable}
Let $X$ be a proper cocompact \cat space.

A closed subgroup $H \se \Isom(X)$ is amenable if and only if the following three conditions hold:
\begin{enumerate}[(1)]
\item $H^\circ$ is soluble-by-compact,

\item $H^\circ \LF(H)$ is open in $H$, 

\item
$H/ (H^\circ \LF(H))$ is virtually soluble.
\end{enumerate}
\end{thmintro}

Thus a closed subgroup $H \se \Isom(X)$ is amenable if and only if it is \{connected soluble\}-by-\{locally elliptic\}-by-\{discrete virtually soluble\}.

\medskip
In the same way as the Tits alternative can be used to find obstructions to linearity, Theorem~\ref{thm:StructureAmenable} implies that many groups cannot appear as closed subgroups of a cocompact group of isometries of a proper \cat spaces. Since amenability passes to the closure, we still get restrictions on arbitrary amenable subgroups. As an extreme example, we recall that there is an active search for (infinite) finitely generated simple amenable groups; non-positively curved spaces will not be the natural habitat where to hunt for them:

\begin{corintro}\label{cor:SimpleGps}
Let $\Gamma$ be an infinite finitely generated amenable group. Assume that the only virtually abelian quotient of $\Gamma$ is the trivial one (e.g. $\Gamma$ is simple).  

Then there is no non-trivial isometric $\Gamma$-action whatsoever on any proper cocompact \cat space.
\end{corintro}

\subsection{Discrete amenable subgroups}
Given  a \emph{discrete} group  $\Gamma $ acting properly and cocompactly on $X$, all the amenable subgroups of $\Gamma$ are virtually abelian and stabilise a flat in $X$: this was proved by Adams--Ballmann (Corollary~B in~\cite{AB98}) and generalises the Solvable Subgroup Theorem (Theorem~II.7.8 in~\cite{Bridson-Haefliger}).

We emphasize that Theorem~\ref{thm:StructureAmenable} \emph{does not suppose the existence of any discrete cocompact group of isometries}: only the (possibly indiscrete) \emph{full} isometry group $\Isom(X)$ is assumed cocompact. Although discrete amenable subgroups of $\Isom(X)$ need not be virtually abelian (for instance the Heisenberg group over $\ZZ$ is a discrete subgroup of $\SL_3(\RR)$), Theorem~\ref{thm:StructureAmenable} states that they must be \{locally finite\}-by-\{virtually soluble\}. We remark that the virtually soluble quotient need not be finitely generated in general (this applies \emph{a fortiori} to the discrete quotient of $H$ in Theorem~
 \ref{thm:StructureAmenable}). Indeed, a \cat lattice such as $\SL_n(\ZZ[1/p])$ contains the infinitely generated abelian group $\ZZ[1/p]$ as a subgroup.

Notice however that every finitely generated subgroup of $\ZZ[1/p]$ is cyclic. This reflects a  general property of $\QQ$-linear soluble groups, all of which have finite \textbf{Pr\"ufer rank}. Recall that this rank is the smallest integer $r$ such that every finitely generated subgroup can be generated by at most $r$ elements. Our next theorem establishes such a finiteness result in the generality of all proper cocompact \cat spaces.

\begin{thmintro}\label{thm:DiscreteAmen}
Let $X$ be a proper cocompact \cat space. Then there is a constant $r= r(X)$ such that the following holds. 

For every discrete amenable subgroup $\Gamma < \Isom(X)$, the quotient $\Gamma/\LF(\Gamma)$ is  virtually \{torsion-free soluble of Pr\"ufer rank~$\leq r$\}.
\end{thmintro}

Remark that for each prime $p$, the lamplighter group  $(\ZZ/p) \wr \ZZ$ can be realised as a discrete group of isometries of a cocompact \cat space, since it embeds in the \cat lattice $\SL_n(\FF_p[t, t\inv])$. Theorem~\ref{thm:DiscreteAmen} implies that, on the other hand, the wreath product $\ZZ \wr \ZZ$ cannot. 

\medskip
Since torsion-free soluble groups of finite Pr\"ufer rank are known to be linear over $\QQ$ by a theorem of Wehrfritz~\cite[pp.25--26]{Wehrfritz}, Theorem~\ref{thm:DiscreteAmen} has the following consequence. 

\begin{corintro}
Let $X$ be a proper cocompact \cat space.

Then any torsion-free discrete amenable subgroup $\Gamma < \Isom(X)$ is $\QQ$-linear. \qed
\end{corintro}

By a theorem of Sh.~Rosset~\cite{Rosset}, the kernel of any homomorphism of a finitely generated group of subexponential growth to a virtually soluble group is itself finitely generated. As pointed out to us by Ami Eisenmann, the latter fact combined with  Theorem~\ref{thm:DiscreteAmen} yields the following, since a virtually soluble group of subexponential growth is virtually nilpotent~\cite{Milnor}. 

\begin{corintro}\label{cor:Ami}
Let $X$ be a proper cocompact \cat space.

Then every finitely generated discrete subgroup $\Gamma < \Isom(X)$ of subexponential growth is virtually nilpotent.

Thus $\Isom(X)$ does not admit discrete subgroups of intermediate growth.
\qed
\end{corintro}

Of course certain groups of intermediate growth can be embedded non-discretely into $\Isom(X)$, for instance if they are residually finite like Grigorchuk's group.

\subsection{Refining spaces}
A fundamental tool for the study of amenable subgroups of isometries of a proper \cat space $X$ is the Adams--Ballmann theorem~\cite{AB98} which states that such a group preserves a flat or fixes a point at infinity. The first case of this alternative is of course highly satisfactory since the isometries of Euclidean space form a very elementary Lie group, but the case of a fixed point at infinity seems at first sight to be of little help for the elucidation of the group.

\smallskip
The way forward here is to use the \textbf{transverse space} $X_\xi$ of a point $\xi\in\bd X$ together with the canonical action of the stabiliser of $\xi$ on $X_\xi$, see Section~\ref{sec:Iter} below (in the classical case of symmetric spaces, this construction goes back at least to the Karpelevich compactification; in general, see also~\cite{Leeb}, \cite{CapraceTD}). This opens the door to iterations, considering \textbf{refining sequences} $(\xi_1, \ldots, \xi_k)$, which are defined by $\xi_1\in\bd X$ and then $\xi_{i+1}\in \bd X_{\xi_1, \ldots, \xi_i}$. A \textbf{refined} point, flat, etc.\ refers to the corresponding object in some $X_{\xi_1, \ldots, \xi_k}$. With this terminology at hand, we can state a converse to the Adams--Ballmann theorem, whose proof relies on iterative constructions of liftings for transverse spaces with (generally non-continuous) associated homomorphisms, and on an appropriate version of the Kazhdan--Margulis theorem.

\begin{thmintro}\label{thm:converseAB}
Let $X$ be a proper cocompact \cat space. 

Then the stabiliser of every refined flat of $X$ is  amenable (and closed). 
\end{thmintro}

This is a converse because both theorems combine to give the following geometric characterisation of all amenable subgroups, thus complementing the algebraic characterisation of Theorem~\ref{thm:StructureAmenable}:

\begin{corintro}\label{cor:converseAB}
Let $X$ be a proper cocompact \cat space.

Then a  closed subgroup of $\Isom(X)$ is amenable if and only if it preserves a refined flat.
\end{corintro}

The `only if' direction follows from a simple iteration of the Adams--Ballmann theorem~\cite{AB98}. This iteration process terminates after finitely many steps because the maximal index $k$ of a refining sequence $(\xi_1, \ldots, \xi_k)$ in $X$ is bounded when $X$ is cocompact. This bound, which we call the \textbf{depth} of $X$, turns out to coincide with the flat rank (see Corollary~\ref{cor:flat_rank} below).

\medskip
There are groups for which every isometric action on a proper \cat space must preserve a refined flat without requesting the amenability of the group. It is easy to produce such examples in a way that runs afoul of the combination of Corollary~\ref{cor:converseAB} and Theorem~\ref{thm:StructureAmenable}, thus giving groups that cannot act at all. Here is an example:

\begin{corintro}\label{cor:T}
Richard Thompson's simple groups $T$ and $V$ do not admit any non-trivial isometric action whatsoever on any proper cocompact \cat space.
\end{corintro}

Recall that Thompson's group $T$ can be viewed as the group of all orientation preserving piecewise affine transformations of $\RR/\ZZ$ which have dyadic breaking points and whose slopes are integral powers of two. Alternatively, it admits the following finite presentation, see~\cite[\S5]{Cannon-Floyd-Parry}.
\begin{align*}
T = \big\langle a,b,c \ \big|\ &[ab^{-1}, a^{-1}ba],\ [ab^{-1},a^{-2}ba^2],\ c^{-1}ba^{-1}cb,\\
&(a^{-1}cba^{-1}ba)^{-1}ba^{-2}cb^2,\ a^{-1}c^{-1}(a^{-1}cb)^2,\ c^3 \big\rangle
\end{align*}
This group is known to have a very nice (proper) isometric action on a locally finite \cat cube complex~\cite{Farley05}, but the latter space is not cocompact -- indeed not finite-dimensional. A similar construction is available for $V$.

At the end of Section~\ref{sec:structure}, we shall give various examples of groups, amenable or not, some of them torsion-free, for which no isometric action on any proper cocompact \cat space can be faithful. Recall that every countable group embeds in a $2$-generated simple group (see~\cite{Hall} and~\cite{Schupp}). Moreover, if the countable group in question is torsion-free, then the simple group can also be chosen torsion-free, as follows from Theorem~A in~\cite{Meier85} (see also Remark~2 on p.~392 in loc.cit.). Applying this embedding theorem to one of the torsion-free groups without faithful action on a proper cocompact \cat space constructed in Section~\ref{sec:structure} below, we obtain the following. 

\begin{corintro}\label{cor:simple}
There is an infinite $2$-generated  torsion-free simple group that  does not admit any non-trivial isometric action on any proper cocompact \cat space. \qed
\end{corintro}

Notice that the assumption that a space be cocompact, which does not impose any restriction on the action of the groups that we consider, does not either imply any sort of local regularity nor of finite dimensionality. For very explicit examples of proper cocompact minimal CAT($-1$) spaces that are infinite-dimensional, see~\cite{Monod-Py}.

\subsection{Fixed points and Levi decompositions}
In the quest for understanding cocompact \cat spaces and their cocompact isometry groups, we have proposed in~\cite{Caprace-Monod_structure} some structural results under the assumption that no point at infinity be fixed by the entire cocompact group. Notice however that the results stated above do not request that hypothesis; their proofs therefore require to study the remaining case, i.e when some point at infinity has a stabiliser acting cocompactly on the space. 
We shall prove that such a stabiliser admits a Levi decomposition generalising the situation of parabolic subgroups of semi-simple algebraic groups; moreover, the corresponding ``Levi factor'' will again be represented as a cocompact isometry group of a suitable \cat subspace.

\medskip
Let thus $X$ be a proper \cat space, $G<\Isom(X)$ a closed subgroup and $\xi\in\bd X$. We define
$$G_\xi^\mathrm{u} := \Big\{\ g\in G \ : \ \lim_{t\to\infty} d\big(g\cdot r(t), r(t)\big)=0 \ \ \forall\, r \text{\ with\ } r(\infty)=\xi \Big\}$$
(where $r\colon \RR\to X$ are geodesic rays). This is a closed normal subgroup of $G_\xi$. We recall that $\Op(\xi)$ is the set of points $\xi'\in\bd X$ that are \emph{visually opposite} $\xi$, i.e. such that there is a  bi-infinite geodesic line in $X$ with extremities $\xi$ and $\xi'$. The union of all such lines is denoted by $P(\xi, \xi')$;  it is a closed convex subspace of $X$ with a non-trivial Euclidean factor. If the stabiliser $\Isom(X)_\xi$ acts cocompactly, then the set $\Op(\xi)$ is non-empty. 

\begin{thmintro}\label{thm:Levi}
Assume that $G_\xi$ acts cocompactly on $X$. Then for each $\xi' \in \Op(\xi)$ we have a  decomposition
$$G_\xi = G_{\xi, \xi'} \cdot G_\xi^\mathrm{u} $$
which is almost semi-direct in the sense that $G_{\xi, \xi'} \cap G_\xi^\mathrm{u} $ is compact. In particular, $G_\xi^\mathrm{u}$ acts transitively on $\Op(\xi)$.
\end{thmintro}

This algebraic Levi decomposition comes with the  geometric counterpart below, which generalises the fact that, for semi-simple Lie groups, the Levi factor acts on a totally geodesic copy of its symmetric space embedded in the ambient space by the Mostow--Karpelevich theorem.

\begin{propintro}\label{prop:Levi:geom}
In the setting of Theorem~\ref{thm:Levi}, the double stabiliser $G_{\xi, \xi'}$ acts cocompactly on $P(\xi, \xi')$ and there is a canonical isometry $P(\xi, \xi')\cong \RR\times X_\xi$ so that $G_{\xi, \xi'}\cap \Ker(\beta_\xi)$ acts cocompactly on $X_\xi$.
\end{propintro}

\noindent
(Herein $\beta_\xi\colon G_\xi\to\RR$ denotes the Busemann character.)

\medskip
As we shall see in Section~\ref{sec:compactible}, the kernel of the $G_\xi$-action on $X_\xi$ is amenable. At this point, the following scheme emerges for a completely general cocompact group $G$ of isometries of a proper \cat space $X$:

\smallskip
If there is no global fixed point at infinity, then some structure results are provided in~\cite{Caprace-Monod_structure}. Otherwise, we have a new cocompact $G$-action on a transverse space $X_\xi$. On the one hand, the kernel of this $G$-action is amenable, and thus described by Theorem~\ref{thm:StructureAmenable}. On the other hand, the remaining quotient of $G$ is again a cocompact group of isometries, and we can therefore repeat this dichotomic analysis. This process terminates in finitely many steps since each transverse space will sit in $X$ with an additional Euclidean factor by Proposition~\ref{prop:Levi:geom}.

\bigskip
Another motivation for Theorem~\ref{thm:Levi} is the study of \cat lattices (in the sense of~\cite{Caprace-Monod_discrete}). Since many of the results in~\cite{Caprace-Monod_structure}, \cite{Caprace-Monod_discrete} depend on the assumption that no point at infinity be fixed simultaneously by all isometries, it was quite valuable to know that \emph{uniform} \cat lattices (a.k.a\ \cat groups) have no fixed point at infinity except possibly on the Euclidean factor (upon passing, if needed, to the canonical minimal invariant \cat subspace). This was essentially established in~\cite{BurgerSchroeder87}, see~\cite[Cor.~2.7]{AB98}. It has been generalised to finitely generated \cat lattices in Proposition~3.15 of~\cite{Caprace-Monod_discrete}. However, none of these arguments seem to apply to infinitely generated lattices; but using Theorem~\ref{thm:Levi} we obtain the desired statement:

\begin{thmintro}\label{thm:lattice:NoFixed}
Let $X$ be a proper \cat space without Euclidean factor and such that $\Isom(X)$ acts cocompactly and minimally.

If $\Gamma<\Isom(X)$ is any lattice, then there are no $\Gamma$-fixed points at infinity.
\end{thmintro}

As it turns out, the relevant consequence of Theorem~\ref{thm:Levi} applies much more generally to \emph{unimodular} groups:

\begin{thmintro}\label{thm:unimodular}
Let $X$ be a proper \cat space and $G<\Isom(X)$ a closed subgroup acting cocompactly and minimally on $X$.

If $G$ is unimodular, then its fixed points at infinity are contained in the boundary of the Euclidean factor.
\end{thmintro}

\subsection*{Location of the proofs}
Theorems~\ref{thm:StructureAmenable}, \ref{thm:DiscreteAmen} and~\ref{thm:converseAB}, as well as their corollaries, are proved in the final section of the paper. They rely on the one hand, on some facts pertaining to  the structure theory of general  locally compact groups, which we  collect in a preliminary Section~\ref{sec:Yamabis}, and on the other hand on specific geometric tools, which are developed in Sections~\ref{sec:Horo} and~\ref{sec:compactible}. 

Section~\ref{sec:Horo} discusses the transverse spaces and refined bordification of $X$. It culminates in a proof of the Levi decomposition (Theorem~\ref{thm:Levi} and Proposition~\ref{prop:Levi:geom}). The amenability of the ``unipotent radical'' in the Levi decomposition is established in Section~\ref{sec:compactible}, using a notion of \emph{compactible subgroups}  that is analogous but quite more general than the \emph{compaction groups} studied in~\cite{CCMT}. Compactibility is some weak form of contractibility, which is confronted to the invariance of the Haar measure under conjugacy in order to deduce Theorems~\ref{thm:lattice:NoFixed} and~\ref{thm:unimodular}. The final Section~\ref{sec:Amen} is devoted to the proofs of the geometric and algebraic characterisations of amenable subgroups.

\subsection*{Acknowledgements}

The final writing of this paper was partly accomplished when both authors were visiting the Mittag-Leffler Institute, whose hospitality was greatly appreciated. Thanks are also due to Ami Eisenmann for pointing out Corollary~\ref{cor:Ami}. 

\section{On the structure of locally compact groups}\label{sec:Yamabis}

\begin{flushright}
\begin{minipage}[t]{0.75\linewidth}\itshape\small
It seems that all the problems concerning the structure of locally compact groups have been completely solved.
\begin{flushright}
\upshape Hidehiko Yamabe,~\cite[p.~352]{Yamabe53}
\end{flushright}
\end{minipage}
\end{flushright}

\vspace{.5cm}
According to the solution to Hilbert's fifth problem, which we shall refer to as \emph{Yamabe's theorem} in the sequel, every locally compact group $G$ such that $G/G^\circ$ is compact has a unique maximal compact normal subgroup $W$ such that $G/W$ is a virtually connected Lie group (see~\cite{Yamabe53} and Theorem~4.6 in~\cite{Montgomery-Zippin}). This plays a fundamental role in the proof of the following.

\begin{thm}\label{thm:Yamabis}
Let $Y$ be a locally compact group such that the group of components $Y/Y^\circ$ is locally elliptic. 

Then $Y/\LF(Y)$ is a Lie group, $Y^\circ \LF(Y)$ is open in $Y$ and the discrete quotient $Y/(Y^\circ \LF(Y))$ is virtually soluble. 
\end{thm}

\noindent
The part of the statement not regarding virtual solubility was obtained in Theorem~A.5 from~\cite{CoTe}.

\medskip
We emphasize that the discrete quotient $Y/(Y^\circ \LF(Y))$ need not be virtually torsion-free. Indeed, this is illustrated by the semi-direct product 
$$Y = \RR^2 \rtimes C_{p^\infty} < \RR^2 \rtimes O(2),$$ 
where  $C_{p^\infty} $ denote the group of all $p^n$-roots of unity with $p$ a prime and $n \geq 0$ an arbitrary integer. In this example the radical $\LF(Y)$ is trivial and the group of components $Y/Y^\circ \cong  C_{p^\infty}$ is abelian and locally elliptic, but not virtually torsion-free. 

\bigskip
The proof of Theorem~\ref{thm:Yamabis} requires some preparation and will be given at the end of this chapter. 

\subsection{Locally compact subgroups of Lie groups}
Following the general convention in the theory of locally compact groups, we define a  \textbf{Lie group} as a locally compact group $G$ such that the identity component $G^\circ$ is open in $G$ and is a connected Lie group in the usual sense.

\begin{prop}\label{prop:Lie}
Let $G$ be a Lie group and $H$ be a locally compact group admitting a continuous faithful homomorphism into $G$. 

Then we have the following.

\begin{enumerate}[(i)]
\item $H$ is a Lie group. 

\item If  $G/G^\circ$ is virtually soluble, then the following assertions are equivalent:
\begin{enumerate}[(a)]
\item $H/H^\circ$ is virtually soluble; 

\item $H/H^\circ$ is amenable;

\item $H/H^\circ$ does not contain non-abelian free subgroups.
\end{enumerate}
\end{enumerate}
\end{prop}

\begin{proof}
(i) is an immediate consequence of the characterization of Lie groups as those locally compact groups having no small subgroups, see~\cite{Montgomery-Zippin}. 

\medskip \noindent (ii)
Assume that $G/G^\circ$ is virtually soluble. The implications (a)~$\Rightarrow$~(b)~$\Rightarrow$~(c) are clear. We assume henceforth that (c) holds. 
Let $\alpha \colon H \to G$ be a continuous faithful homomorphism, let $B = \overline{\alpha(H)}$ and $A = \overline{\alpha(H^\circ)}$. Thus $A$ is a closed connected normal subgroup of $B$, hence $ A \leq B^\circ$.  


By Theorem~2.1 from Chapter~XVI in~\cite{Hochschild}, the quotient of a connected Lie group by a dense connected normal subgroup is abelian. This implies that $A/\alpha(H^\circ)$ is abelian. So is thus $H^1/H^\circ$, where $H^1 = \alpha^{-1}(A)$. Let also $H^2 = \alpha^{-1}(B^\circ)$. Then the discrete group $H^2/H^1$ embeds in the connected Lie group $B^\circ /A$. By hypothesis $H^2/H^1$ does not contain non-abelian free subgroups. It must therefore be virtually soluble by the Tits alternative~\cite{Tits72}. Thus $H^2/H^\circ$ is virtually soluble, and it remains to show that $H/H^2$ is virtually soluble. Since the latter embeds continuously and faithfully in $B/B^\circ$, it suffices to prove that $B/B^\circ$ is virtually soluble.  In other words, we have reduced the problem to the special case when $H=B$ is a closed subgroup of a Lie group $G$ with $G/G^\circ$ virtually soluble. This is what we assume henceforth. 

Since $H^\circ $ is contained in $G^\circ$, there is no loss of generality in assuming that $G$ is connected. 


Let $R$ be the soluble radical of $H^\circ$ and $S = H^\circ /R$. Let $J$ be the inverse image in $H$ of $\centra_{H/R}(S)$. Since $J$ is an extension of $R$ by a subgroup of $H/H^\circ$, it follows that $J$ does not contain non-abelian free subgroups. Moreover, since   the outer automorphism group of the semisimple Lie group $S$ is finite, it follows that the image of $J$ in $H/H^\circ$ is of finite index. By the Tits alternative, a subgroup of a connected Lie group which does not have free subgroups must be virtually soluble. Thus $J$, and hence also $H/H^\circ$, is virtually soluble, as desired.
\end{proof}

\subsection{Central extensions of locally elliptic groups}

\begin{prop}\label{prop:GenUshakov}
Let $G$ be a locally compact group and $Z < G$ be a closed subgroup contained in the centre $\centra(G)$.

If $G/Z$ is locally elliptic, then $G/\LF(G)$ is abelian (and torsion-free).  
\end{prop}

Proposition~\ref{prop:GenUshakov} is a companion to the following classical result.

\begin{thm}[U{\v{s}}akov~\cite{Ushakov}]\label{thm:Ushakov}
Let $G$ be a locally compact group in which every conjugacy class has compact closure. 

Then the union of all compact subgroups of $G$ forms a closed normal subgroup, which therefore coincides with
$\LF(G)$, and the corresponding quotient $G/\LF(G)$ is abelian (and torsion-free).\qed
\end{thm}

\begin{proof}[Proof of Proposition~\ref{prop:GenUshakov}]
We are given a short exact sequence $1 \to Z \to G \to Q \to 1$ with $Z$ central in $G$ and $Q$ locally elliptic. 
Let $G' = G/\LF(G)$ and $Z'$ be the closure of the image of $Z$ in $G'$; write $Q'=G'/Z'$ and consider the short exact sequence 
$$1 \to Z' \to G' \to Q' \to 1.$$
The normal subgroup $Z'$ is central in $G'$ (since the image of $Z$ is dense) and the quotient $Q'$ is locally elliptic (since it is a quotient of $Q$). Thus the latter short exact sequence satisfies the hypotheses of the Proposition. We need to show that $G'$ is abelian.

\smallskip
Let now $g_1, g_2 \in G'$. Then $\{g_1, g_2\}$ is contained in a group $H$ which an extension of $Z'$ by a compact group since $Q'$ is locally elliptic. Thus every conjugacy class in $H$ has compact closure, and it follows from Theorem~\ref{thm:Ushakov} that $H/\LF(H)$ is abelian and torsion-free. 

We first specialize this observation to the case when $g_1, g_2$ are both elliptic. It then follows that $g_1, g_2 \in \LF(H)$, and thus that $g_1, g_2$ belong to a common compact subgroup of $G'$. Therefore  the set of elliptic elements of $G'$ is a subgroup, and must therefore be contained in $\LF(G')=1$. It follows that $G'$ has no nontrivial compact subgroup. 

Applying again the observation above to two arbitrary elements $g_1, g_2$, and taking now into account  that $\LF(H)$ is trivial (since $G'$ has no nontrivial compact subgroup), we infer that $H$ is abelian and, hence, that $g_1$ and $g_2$ commute. Thus $G'$ is abelian, as desired. 
\end{proof}

\subsection{Connected quotients and Proof of Theorem~\ref{thm:Yamabis}}
The last tool that we need is the following fact; for the proof, we refer to Lemma~2.4 in~\cite{CCMT}.

\begin{lem}\label{lem:quot-Lie}
Let $G$ be a locally compact group with a quotient map $\pi \colon G\twoheadrightarrow L$ onto a Lie group $L$. Then $\pi(G^\circ)=L^\circ$.\qed
\end{lem}

As pointed out in~\cite{CCMT}, this statement can fail if $L$ is not Lie; for an elementary substitute in the general case, see Lemma~\ref{lem:image_conn} below.

\begin{proof}[Proof of Theorem~\ref{thm:Yamabis}]
Let $Y$ be a locally compact group such that $Y/Y^\circ$ is locally elliptic. In order to prove the Theorem, we may assume without loss of generality that $\LF(Y^\circ)=1$. Thus $Y^\circ$ is a connected Lie group by Yamabe's theorem. Theorem~A.5 from~\cite{CoTe} then yields that $Y^\circ \LF(Y)$ is open in $Y$. In particular $Y/\LF(Y)$ is a Lie group. 

It remains to show that $Y/(Y^\circ \LF(Y))$ is virtually soluble. 

Let $\pi \colon Y \to \Aut(Y^\circ)$ be the homomorphism induced by the conjugation action of $Y$ on $Y^\circ$. Since $\Aut(Y^\circ)$ acts faithfully on the Lie algebra of $Y^\circ$, we may assume that $\pi$ takes its values in $\mathrm{GL}_n(\RR)$ for some $n$. By Proposition~\ref{prop:Lie}, the group $\pi(Y)$ (endowed with the quotient topology from $Y/\Ker(\pi)$) is a Lie group. Consequently Lemma~\ref{lem:quot-Lie} yields $\pi(Y^\circ) = \pi(Y)^\circ$. In particular the group of components $\pi(Y)/\pi(Y)^\circ$ is a quotient of $Y/Y^\circ$, and is thus locally elliptic, hence amenable. Invoking Proposition~\ref{prop:Lie} again, it follows that $\pi(Y)/\pi(Y)^\circ$ is virtually soluble. 

Since $\Ker(\pi) = \centra_Y(Y^\circ) >  \LF(Y)$, all it remains to show is that 
$$
Y^\circ \centra_Y(Y^\circ)/ Y^\circ \LF(Y) \cong \centra_Y(Y^\circ) / Y^\circ \LF(Y) \cap \centra(Y^\circ) 
$$
is virtually soluble. Observing further that 
$$
Y^\circ \LF(Y) \cap \centra(Y^\circ)  = (Y^\circ \cap  \centra_Y(Y^\circ)) \LF(Y) = \centra(Y^\circ) \LF(Y),
$$ 
the remaining  statement to be proven is that $\centra_Y(Y^\circ) / \centra(Y^\circ) \LF(Y)$ is virtually soluble. 

Notice that $Z = \centra(Y^\circ)$ is central in $\centra_Y(Y^\circ)$. Moreover the quotient $\centra_Y(Y^\circ)/Z \cong Y^\circ \centra_Y(Y^\circ)/ Y^\circ$ is locally elliptic, as it is isomorphic to a subgroup of $Y/Y^\circ$. We are thus in a position to apply Proposition~\ref{prop:GenUshakov}  to the group $\centra_Y(Y^\circ)$. This implies that $\centra_Y(Y^\circ)/ \LF(\centra_Y(Y^\circ))$ is abelian. Hence so are the groups $\centra_Y(Y^\circ)/ \LF(Y)$ and $ \centra_Y(Y^\circ) / \centra(Y^\circ) \LF(Y)$. This concludes the proof.
\end{proof}

\section{Horoactions and Levi decompositions}\label{sec:Horo}

\subsection{Transverse spaces and horoactions}

To each point $\xi \in \bd X$  at infinity of a \cat space $X$, one associates a new \cat space $X_\xi$ called the \textbf{transverse space} of $\xi$ and defined as follows. Consider the set $X^*_\xi$ of all rays $r\colon\RR_+\to X$ pointing to $\xi$. We claim that the infimal distance
$$d_\xi(r, r')\ = \ \inf_{t, t'\geq 0} d(r(t), r'(t'))$$
between rays is a pseudometric. Indeed, this follows from the convexity of the distance: more precisely, there is a constant $C$ (which is none other than the difference between the Busemann functions associated to $r$ and $r'$) such that
\begin{equation}\label{eq:d_xi}
d_\xi(r, r')\ = \ \lim_{t\to\infty} d(r(t), r'(t+C))
\end{equation}
holds. Now we define $X_\xi$ to be the (Hausdorff) metric completion of the
pseudometric space $X^*_\xi$, still denoting the resulting metric by
$d_\xi$. The formula~(\ref{eq:d_xi}) shows that $X_\xi$ is a \cat
space and there is a canonical $1$-Lipschitz map $X\to X_\xi$ with
dense image.


\medskip

Any isometry $g$ of $X$ induces an isometry $X_\xi \to X_{g\xi}$. In particular, there is a canonical isometric action of the stabiliser $\Isom(X)_\xi$ on $X_\xi$.

\begin{defn}
The isometric $\Isom(X)_\xi$-action on $\RR\times X_\xi$ given by
$$\begin{array}{rcll}
\omega_\xi\colon &G_\xi &\lra& \Isom\big(\RR\times X_\xi\big)\\
\omega_\xi(g)\colon &(t, x) &\longmapsto& (t + \beta_\xi(g), g.x)
\end{array}$$
is called the \textbf{horoaction}. Here, $\beta_\xi(g)$ is the \textbf{Busemann character} associated to $\xi$, which we recall can be expressed as
$$\beta_\xi(g) = \lim_{t\to\infty} \Big( d(r(t), x_0) - d(r(t), g x_0) \Big),$$
where $r$ is any ray pointing to $\xi$ and $x_0$ any point of $X$.
\end{defn}

\begin{remark}\label{rem:horokernel}
This definition shows that the kernel $\Ker(\omega_\xi)$ of the horoaction is precisely the normal subgroup $G_\xi^\mathrm{u} \lhd G_\xi$ introduced before the statement of Theorem~\ref{thm:Levi}.
\end{remark}

Some basic properties of the transverse space $X_\xi$ are collected in the following proposition. A point $\xi\in\partial X$ is called a \textbf{cocompact} point at infinity if its stabiliser $\Isom(X)_\xi$ acts cocompactly on $X$.

\begin{prop}\label{prop:TransverseSpace}
Let $\xi \in X_\xi$. We have the following.
\begin{enumerate}[(i)]
\item $X_\xi$ is a complete \cat space. 

\item The canonical homomorphism $\Isom(X)_\xi \to \Isom(X_\xi)$ is continuous.\label{pt:TransverseSpace:cont}

\item If $X$ is of bounded geometry, then so is $X_\xi$; in particular, $X_\xi$ is proper in that case.

\item \label{it:FiniteDepth} If $X$ is of bounded geometry, then it has finite depth.

\item\label{it:cocompact} If a subgroup of $\Isom(X)_\xi$ acts cocompactly on $X$, then it acts cocompactly on $X_\xi$. In particular, $X_\xi$ is a cocompact space whenever $\xi$ is a cocompact point.
\end{enumerate}
\end{prop}

\begin{proof}
For (i), (ii) and (iii), see Proposition~4.3 from~\cite{CapraceTD}. As for~(\ref{it:FiniteDepth}), it follows from Corollary~4.4 in~\cite{CapraceTD}. The last item is straightforward.
\end{proof}

\begin{remark}\label{rem:disc}
Given a proper \cat space $X$ and a closed subgroup $G  \se \Isom(X)$ fixing a point $\xi \in \bd X$, we warn the reader that the image $\omega_\xi(G)$ of $G$ under the horoaction need not be closed in $\Isom(\RR \times X_\xi)$. Indeed, consider the \cat space $X = H \times T$, defined as  the product of the hyperbolic plane $H =  \HH^2$ and the $p+1$-regular tree $T$, where $T$ is  viewed as the Bruhat--Tits tree of $\SL_2(\QQ_p)$. Let $G \se \Isom(X)$ be the infinite cyclic group 
$G =\{ \left(\begin{array}{cc}
1 & z\\
0 & 1
\end{array}\right) \; |  \; z \in \ZZ\}$ diagonally embedded in $\SL_2(\RR) \times \SL_2(\QQ_p) \se \Isom(X)$. Then $G$ is discrete, hence closed in $\Isom(X)$. It fixes a point $\xi$ at infinity of the $\HH^2$-factor, annihilates the Busemann character $\beta_\xi$ associated to that point, and fixes a point $x \in T$ in the tree factor.  

Since $\xi$ lies on the boundary of the factor $H$, we have $X_\xi = (H \times T)_\xi  = H_\xi \times T \cong T $ since $H_\xi$ is reduced to a singleton because $H$ is CAT($-1$). In other words there is a canonical isometry  $T \to X_\xi$   which is $\Isom(X)_\xi$-equivariant. 
In particular the closure of the image $\omega_\xi(G)$ of $G$ under the horoaction in  $\Isom(\RR \times X_\xi)$ is isomorphic to the closure of $G$ in $\SL_2(\QQ_p)$, and is thus compact and infinite. In particular $\omega_\xi(G)$ is  not closed. 
\end{remark}

\subsection{Lifting the transverse space}

\begin{flushright}
\begin{minipage}[t]{0.75\linewidth}\itshape\small
\ldots quelque chose d'aussi na\"if que l'entreprise d'atteindre l'horizon en marchant devant soi.
\begin{flushright}
\upshape Marcel Proust, \emph{\`A la recherche du temps perdu -- Albertine disparue}, 1927
\end{flushright}
\end{minipage}
\end{flushright}

\vspace{.5cm}

Our next goal is to construct an isometric section $X_\xi\to X$ for the projection $X\to X_\xi$ when $\xi$ is a cocompact point at infinity (or opposite a cocompact point). We first record in Proposition~\ref{prop:lift:facile} below an elementary observation which is valid under a much weaker assumption.

\begin{defn}\label{defn:radial}
Let $G$ be any group acting by isometries on a \cat space $X$. We call $\xi\in \bd X$ a \textbf{radial limit point} (of $G$) if there is a sequence $\{g_n\}$ in $G$ such that for some ray $r_0\colon\RR^+\to X$ pointing to $\xi$, the sequence $g_n\inv r_0(0)$ converges to $\xi$ while remaining at bounded distance of $r_0$.  
\end{defn}

Notice that in that case, the condition holds for all rays $r$ pointing to $\xi$. In fact, $g_n\inv x_0$ converges to $\xi$ while remaining at bounded distance of $r$ for any $x_0\in X$. Therefore we shall simply call such a sequence $\{g_n\}$ \textbf{radial for $\xi$}.

\begin{remarks}
\ 
\begin{enumerate}
\item If $X$ is cocompact, then any $\xi\in\bd X$ is a radial limit point.
\item This definition should be distinguished from certain notions of \emph{conical} limit points that are not equivalent in the absence of negative curvature (for instance, the terminology of~\cite{Hattori} is incompatible with that of~\cite{Albuquerque}).
\item When no $G$ is given, it is always understood that $G=\Isom(X)$.
\end{enumerate}
\end{remarks}

For radial limit points, a simple limiting argument provides the following weak form of lifting, to be strengthened for cocompact points in Theorem~\ref{thm:lifting} below.

\begin{prop}\label{prop:lift:facile}
Let $X$ be a proper \cat space. For every radial limit point $\xi\in\bd X$, there exists an isometric embedding $\RR\times X_\xi\to X$.
\end{prop}

The proposition has the following immediate consequence for the geometric dimension of the Tits boundary; a stronger statement for cocompact spaces will be given in Corollary~\ref{cor:flat_rank} below.

\begin{cor}\label{cor:dim_chute}
If $\xi\in\bd X$ is a radial point of the proper \cat space $X$, then $\dim(\bd X_\xi) \leq \dim(\bd X)- 1$.\qed
\end{cor}

\begin{proof}[Proof of Proposition~\ref{prop:lift:facile}]
Let $r_0$, $\{g_n\}$ be as in Definition~\ref{defn:radial}. Let $t_n\geq 0$ be such that $r_0(t_n)$ is the projection of $g_n\inv  r_0(0)$ to $r_0$. Then, for any given $r$ in $X^*_\xi$, the sequence $g_n (r(t_n))$ remains bounded in $X$. Therefore there is a simultaneous accumulation point for the family (indexed by $r \in X^*_\xi$) of sequences of isometric maps
$$[-t_n, \infty) \lra X,\kern10mm t\longmapsto g_n \big(r(t+t_n)\big),$$
uniformly on bounded intervals. This accumulation point is therefore a map $f^*$ from $X^*_\xi$ to the set of isometric maps $\RR\to X$. By construction, any two maps in the image of $f^*$ remain at bounded distance. More precisely, the (Hausdorff) distance between $f^*(r)(\RR)$ and $f^*(r')(\RR)$ is exactly $d_\xi(r, r')$ for any $r, r'\in X^*_\xi$. Denote by $T$ the metric space of all lines at finite distance of $f^*(r_0)(\RR)$ endowed with Hausdorff distance.  The union of all these lines is a convex subset of $X$ isometric to the product $\RR\times T$, see~\cite[II.2.14]{Bridson-Haefliger}; in fact this union is closed, i.e.\ $T$ is complete (compare e.g.\ the last sentence before Remark~40 in~\cite{Monod_superrigid}). The map $f^*\colon X^*_\xi\to T$ preserves (pseudo-)distances; therefore, $X_\xi$ is isometric to a subset of $T$, finishing the proof.
\end{proof}

Following~\cite{Caprace-Monod_structure}, we say that a point $\xi' \in \bd X$ is \textbf{opposite}  to $\xi$ if there is a geodesic line in $X$ whose extremities are $\xi$ and $\xi'$.  We say that $\xi' $ is \textbf{antipodal}  to $\xi$ if $\tangle \xi {\xi'} = \pi$. Thus the set $\Op(\xi)$ of points opposite to $\xi$ is contained in the set $\Ant(\xi)$ consisting of all antipodes of $\xi$. It is shown in~\cite[Proposition~7.1]{Caprace-Monod_structure} that if $\Isom(X)_\xi$ acts cocompactly, then $\Op(\xi)$ is non-empty and any closed cocompact subgroup of $\Isom(X)_\xi$ acts transitively on $\Op(\xi)$.

\smallskip

Given $\xi' \in \Op(\xi)$, we denote by $P(\xi, \xi')$ the union of all geodesic lines joining $\xi$ to $\xi'$. The set $P(\xi, \xi')$ is closed and convex in $X$; moreover there is a canonical isometric identification
\begin{equation}\label{eq:P}
P(\xi, \xi') \cong \RR \times Y
\end{equation}
for some complete \cat space $Y$ (see~\cite[II.2.14]{Bridson-Haefliger}). When $\xi$ is cocompact, the following Theorem identifies $Y$ with $X_\xi$ and in particular provides an isometric section $X_\xi\to X$ for the canonical map $X\to X_\xi$.

\begin{thm}\label{thm:lifting}
Assume that $\xi$ is cocompact.

Then, for any $\xi'\in\Op(\xi)$, the factor $Y$ in~(\ref{eq:P}) is canonicallly isometrically identified with $X_\xi$ under the map $X\to X_\xi$. Thus, there is a canonical isometric identification $P(\xi, \xi') \cong \RR \times X_\xi $ which is $\Isom(X)_{\xi, \xi'}$-equivariant, the space $\RR \times X_\xi $ being endowed with the horoaction.
\end{thm}

\begin{proof}
Let $r_0\colon \RR\to X$ be a geodesic line with $r_0(-\infty)=\xi'$ and $r_0(+\infty)=\xi$. Since $\xi$ is cocompact, there is a sequence $\{g_n\}$ in $\Isom(X)_\xi$ such that $d(g_n .r_0(0), r_0(-n))$ remains bounded. By convexity, $g_n (r_0(n))$ remains bounded in $X$ and therefore we are in particular in the situation of the proof of Proposition~\ref{prop:lift:facile}. We thus obtain an isometric map
$$\RR \times X_\xi \lra P(\xi, \xi') \cong \RR \times Y \ \se\ X$$
preserving the product decomposition, yielding in particular a (non-canonical) isometric embedding $\iota\colon X_\xi\to Y$. On the other hand, the restriction of $X\to X_\xi$ to $P(\xi, \xi')$ induces a canonical isometric map $\teta\colon Y\to X_\xi$ by the definition of $X_\xi$. Now $\teta\circ\iota$ is an isometric self-map of the space $X_\xi$ which is cocompact by Proposition~\ref{prop:TransverseSpace}(\ref{it:cocompact}). It follows that this self-map is onto (Proposition~4.5 in~\cite{Caprace-Sageev}) and thus $\teta$ is onto too. 

The action of $\Isom(X)_{\xi, \xi'}$ preserves the decomposition~(\ref{eq:P}) by construction and thus is a product action, see~\cite[I.5.3(4)]{Bridson-Haefliger}. The formula with $\beta_\xi(g)$ follows by direct computation.
\end{proof}

The cocompactness of $\xi$ does not imply the cocompactness of its opposites (as illustrated by Heintze manifolds~\cite{Heintze74}). Nevertheless, these opposites still enjoy the conclusions of Theorem~\ref{thm:lifting}. To make this statement precise, observe that the canonical isometry  $P(\xi, \xi') \cong \RR \times X_\xi $ induces a canonical isometric embedding $X_\xi \to X_{\xi'}$.

\begin{prop}\label{prop:OppositePair}
Let $\xi'$ be any point opposite a cocompact point $\xi$. Then the canonical isometric embedding $X_\xi \to X_{\xi'}$ is an isometry. In particular, $P(\xi, \xi') \cong \RR \times X_{\xi'}$ and this identification is compatible with the horoaction $\omega_{\xi'}$ of $\Isom(X)_{\xi, \xi'}$.
\end{prop}

\begin{proof}
Keep the notation of the proof of Theorem~\ref{thm:lifting} and consider the ray $r^-_0\colon \RR^+\to X$ defined by $r^-_0(t) = r_0(-t)$. Observe that the sequence $\{g_n\inv\}$ is as required in the proof of Proposition~\ref{prop:lift:facile} with $\xi'$ instead of $\xi$. We thus obtain an isometric embedding $\RR\times X_{\xi'}\to X$ such that one endpoint of the $\RR$-factor is $\xi$; we denote the opposite endpoint by $\xi''$. We therefore have an isometric embedding of $\RR\times X_{\xi'}\to X$ into $P(\xi'', \xi) \cong \RR\times Z$ (for some $Z$) preserving the product structures. Moreover, Theorem~\ref{thm:lifting} provided canonical isometries
$$P(\xi', \xi)\ \cong\ P(\xi'', \xi)\ \cong\ \RR\times X_\xi$$
preserving the product structures. At this point we have mutual isometric embeddings of $\RR\times X_{\xi'}$ and $\RR\times X_{\xi}$ into each other and conclude as in the proof of Theorem~\ref{thm:lifting} since $\RR\times X_{\xi}$ is cocompact by Proposition~\ref{prop:TransverseSpace}(\ref{it:cocompact}).
\end{proof}

\begin{remark}
It follows from Theorem~\ref{thm:lifting} and Proposition~\ref{prop:OppositePair} that if $\xi \in \bd X$ is a cocompact point, or if it is opposite a cocompact point, then the (non-Hausdorff) space $X^*_\xi$ is already complete and the map $X\to X_\xi$ is onto. 
\end{remark}

\subsection{Existence of a `Levi decomposition'}\label{sec:Levi}

The following result will be our tool to investigate large stabilisers of points at infinity. Incidentally, it refines Proposition~7.1 in~\cite{Caprace-Monod_structure} which stated that if $G_\xi$ acts cocompactly on $X$, then it acts transitively on $\Op(\xi)$.

\begin{thm}\label{thm:Levi:bis}
Let $G <\Isom(X)$ be a closed subgroup and assume $G_\xi$ acts cocompactly on $X$. Let $N<G_\xi$ be a subgroup that is normalised by some radial sequence $\{h_n\}$ in $G_\xi$ for $\xi$. Let $\xi'\in\bd X$ be a limit point of $\{h_n\}$; hence $\xi' \in \Op(\xi)$.

Then, writing  $N^\mathrm{u}=N\cap \Ker(\omega_\xi)$ where $\omega_\xi$ is the horoaction, we have a decomposition
$$N = N_{\xi'} \cdot N^\mathrm{u} $$
which is almost semi-direct in the sense that $N_{\xi'} \cap N^\mathrm{u} $ has compact closure. In particular $N^\mathrm{u}$ acts transitively on the $N$-orbit of $\xi'$ in
$\Op(\xi)$.
\end{thm}

The theorem applies in particular when $N$ is a normal subgroup of $G_\xi$. In that case, the above decomposition hold for \emph{any} $\xi' \in \Op(\xi)$ since $G_\xi$ acts transitively on $\Op(\xi)$ by Proposition~7.1 in~\cite{Caprace-Monod_structure}. Thus we obtain Theorem~\ref{thm:Levi} as the case $N=G_\xi$, recalling Remark~\ref{rem:horokernel} for the notation $G_\xi^\mathrm{u}$.

\bigskip

We shall carry out the first part of the argument in a more general setting: \emph{for the time being, we only assume that $\xi$ is a radial limit point of $G$}.

\smallskip

Let $r_0$ be a ray pointing to $\xi$ and $\{g_n\}$ be any radial sequence as in Definition~\ref{defn:radial}. At this point there is no additional restriction on $\{g_n\}$, though for Theorem~\ref{thm:Levi:bis} we will later choose $g_n=h_n$. As in the proof of Proposition~\ref{prop:lift:facile}, we have (upon passing to subsequences) a sequence of numbers $t_n\geq 0$ such that the sequence of maps
$$[-t_n, \infty) \lra X,\kern10mm t\longmapsto g_n \big(r_0(t+t_n)\big)$$
converges to a geodesic line $\sigma\colon\RR\to X$ uniformly on bounded intervals. In particular, $g_n\xi$ converges to the point $\eta:=\sigma(\infty)$ and $g_nr_0(0)$ to $\eta':=\sigma(-\infty)$. We denote by $f\colon\RR\times X_\xi\to X$ the isometric embedding constructed in Proposition~\ref{prop:lift:facile} and observe that $f$ ranges in $P(\eta, \eta')$ by construction.

A convexity argument shows that the sequence $\Ad(g_n)(g)=g_n g g_n\inv$ remains bounded in $G$ whenever $g\in G_\xi$. Therefore, we can chose a map $\ro\colon G_\xi\to G$ that is a (point-wise) cluster point of the sequence $\Ad(g_n)|_{G_\xi}$. By construction, $\ro$ is a group homomorphism; it need not be continuous, see Example~\ref{ex:ro} below.

\begin{prop}\label{prop:ro}
In the above setting, the following hold.
\begin{enumerate}[(i)]
\item $\ro(G_\xi)\se G_{\eta, \eta'}$,\label{ro:pt:image}
%
%
\item $\Ker(\ro) \se \Ker(\omega_\eta\circ \ro) \se \Ker(\omega_\xi)$,\label{ro:pt:noyau}
%
%
\item $\beta_\eta \circ \ro = \beta_\xi$,\label{ro:pt:char}
\item The isometric embedding $f\colon\RR\times X_\xi\to P(\eta, \eta')$ is $\ro$-equivariant for $G_\xi$.\label{ro:pt:equi}
\end{enumerate}
\end{prop}

\begin{proof}
Points~(\ref{ro:pt:image}) and~(\ref{ro:pt:equi}) follow from the construction of $\ro$ and $f$. Likewise for point~(\ref{ro:pt:char}), using that the pointwise convergence on Busemann functions corresponds to the cone topology on the boundary.

Turning to~(\ref{ro:pt:noyau}), let $g \in G_\xi$ be an element acting non-trivially on $X_\xi$. 
Then there is a ray $r_1$ asymptotic to $r_0$ and $\vareps>0$ such that
$d(g r_1(t), r_1(s))\geq \vareps$ for all $s,t\geq 0$. This implies that $\ro(g)$ moves a line parallel to
$\sigma(\RR)$ by at least $\vareps$. Thus $\omega_\eta(\ro(g)) \neq 1$, which shows
that $\Ker(\omega_\eta\circ\ro)$ acts trivially on the transverse space $X_\xi$. Together with~(\ref{ro:pt:char}), this establishes the non-trivial inclusion of~(\ref{ro:pt:noyau}).
\end{proof}

\begin{proof}[Proof of Theorem~\ref{thm:Levi:bis}]
We now assume that $G_\xi$ is cocompact and consider the radial sequence $\{h_n\}$ in $G_\xi$ given by hypothesis. We first observe that the image of $\{h_n\}$ in $\Isom(X_\xi)$ is bounded and hence we can assume that it converges to some $\alpha\in \Isom(X_\xi)$.  We now specialise the above discussion to the case $g_n=h_n$. In particular, $\eta=\xi$ and $\eta'=\xi'$.

The sequence $\Ad(\omega_\xi( h_n))$ converges to $\Id \times \Ad(\alpha)$ as automorphisms of $\Isom(\RR\times X_\xi)$ and hence the diagram
$$\xymatrix{
G_\xi \ar[rr]^{\ro} \ar[d]_{\omega_\xi} && G_{\xi, \xi'}\ar[d]^{\omega_\xi}\\
\Isom(\RR\times X_\xi)\ar[rr]_{\Id \times \Ad(\alpha)} && \Isom(\RR\times X_\xi)
}$$
is commutative. Since each $h_n$ normalises $N$, the automorphism $\Id \times \Ad(\alpha)$ preserves $\omega_\xi(N)$. We deduce that $\omega_\xi(N) =\omega_\xi(\ro(N)) \se \omega_\xi(N_{\xi'})$ and therefore $\omega_\xi(N) = \omega_\xi(N_{ \xi'})$. It follows that we have
$$N = N_{\xi'} \cdot \Ker(\omega_\xi|_N).$$
This provides the desired decomposition of $N$.

\smallskip
For the compactness statement about $N_{\xi'} \cap N^\mathrm{u}$, we claim that this intersection acts trivially on $P(\xi, \xi')$. Indeed, $N_{\xi'}$ preserves $P(\xi, \xi')$ and the horoaction of $N^\mathrm{u}$ on $\RR\times X_\xi$ is trivial; therefore, the claim follows from Theorem~\ref{thm:lifting}.
\end{proof}

\begin{proof}[Proof of Proposition~\ref{prop:Levi:geom}]
Pick $p_0 \in P(\xi, \xi')$ and let $R>0$ be such that $X = G_\xi.B(p_0, R)$. Given $p \in P(\xi, \xi')$, there is $g \in G_\xi$ such that $d(g.p, p_0) < R$. By Proposition~7.1 in~\cite{Caprace-Monod_structure}, $G_\xi$ acts transitively on $\Op(\xi)$; therefore, the decomposition of Theorem~\ref{thm:Levi:bis} applied to $N=G_\xi$ shows that $G^\mathrm{u}_\xi = G \cap \Ker(\omega_\xi)$ also acts transitively. Thus there exists $k \in G^\mathrm{u}_\xi$ be such that $kg.\xi' = \xi'$; in particular we have $kg \in G_{\xi, \xi'}$. Since both $p_0$ and $kg.p$ belong to $P(\xi, \xi')$, we deduce from Theorem~\ref{thm:lifting} that
$$\begin{array}{rcl}
d(p_0, kg.p)^2 & = & (b_\xi(p_0) - b_\xi(kg.p))^2 + d(\pi_\xi(kg.x), \pi_\xi(p_0))^2\\
& = &  (b_\xi(p_0) - b_\xi(g.p))^2 + d(\pi_\xi(g.x), \pi_\xi(p_0))^2\\
& \leq & d( g.x, p_0)^2 + d(g.x, p_0)^2\\
& < & 2 R^2,
\end{array}$$
where $\pi_\xi$ denotes the $G_\xi$ action on $X_\xi$. Thus we have $P(\xi, \xi') \se G_{\xi, \xi'}.B(p_0, \sqrt 2 R)$, which confirms that $G_{\xi, \xi'}$ acts cocompactly on $P(\xi, \xi')$. The additional statements now follow by applying Theorem~\ref{thm:lifting}.
\end{proof}

The following example shows that $\ro$ need not be continuous.

\begin{example}\label{ex:ro}
Consider the simplicial line $L$ whose vertex set $(v_n)_{n \in \ZZ}$ is linearly ordered by the integers. Let $T$ be the tree obtained by adding to $L$ a collection of vertices of valency one, say $\{v'_n, v''_n \; : \; n \in \ZZ\}$, where $v'_n$ and $v''_n$ are declared adjacent to $v_n$. We view $T$ as a metric tree with all edges of length one. Thus $T$ is a \cat space with two endpoints, say $\xi$ and $\xi'$. The isometry group $G = \Isom(T)$ acts cocompactly. Let $t \in G$ be the translation defined by $t \colon (v_n, v'_n, v''_n) \mapsto (v_{n-1}, v'_{n-1}, v''_{n-1})$ for all $n$. We have  $G_\xi = G_{\xi, \xi'} \cong (\prod_{\ZZ} \ZZ/2) \rtimes \ZZ$, where the cyclic factor $\ZZ$ is the group generated by $t$.

Let now $\ro\colon G_\xi\to G$ be a limit point of the sequence $\Ad t^n|_{G_\xi}$. We claim that $\ro$ is not continuous.

Indeed, let $f_n \in G$ be  defined by
$$
f_n \colon  (v_m, v'_m, v''_m) \mapsto
\left\{
\begin{array}{ll}
(v_{m}, v'_{m}, v''_{m}) & \text{if } m < n\\
(v_{m}, v''_{m}, v'_{m}) & \text{if } m \geq n
\end{array} \right.
$$
and set $f_\infty \colon (v_m, v'_m, v''_m) \mapsto (v_{m}, v''_{m}, v'_{m})$ for all $m$.
Then $\lim_n f_n = \Id$. On the other hand, for all $n$ and we have $ \lim_k  t^{k} f_n  t^{-k} = f_\infty$ so that $\rho(f_n) = f_\infty \neq \Id$. In particular   $\ro$ is discontinuous at the identity.
\end{example}

Notice however that, although $\ro$ is not continuous, the composite map $\omega_\xi \circ \ro $ is always continuous, in view of the commutative diagram from the proof of Theorem~\ref{thm:Levi:bis}, recalling that the horoaction is continuous by Proposition~\ref{prop:TransverseSpace}(\ref{pt:TransverseSpace:cont}).

\subsection{Iterations and the refined bordification}\label{sec:Iter}
The concept of transverse space suggests iterative constructions; this is  formalised in the following definition.

\begin{defn}
A \textbf{refining sequence} for the \cat space $X$ is a sequence $(\xi_1, \ldots , \xi_k)$ with $k\geq 0$, $\xi_1\in\bd X$ and $\xi_{i+1}\in \bd X_{\xi_1, \ldots, \xi_i}$ for $i\geq 1$ ($k=0$ corresponds to the empty sequence). The \textbf{refined bordification} of $X$ is the set of \textbf{refined points} $(\xi_1, \ldots , \xi_k; x)$, wherein $(\xi_1, \ldots , \xi_k)$ is a refining sequence and $x\in X_{\xi_1, \ldots, \xi_k}$. Thus, $k=0$ corresponds simply to a point $x\in X$. We call $k$ the \textbf{depth} of the refined point. The \textbf{depth} of the space $X$ is the supremum of the depths of its refined points. (A characterisation of the depth of cocompact spaces will be given in Corollary~\ref{cor:flat_rank} below.)
\end{defn}

Thus the space $X$ is contained in the refined bordification, while $\bd X$ is a quotient of it.

\medskip
Our next goal is to iterate the constructions described in the previous sections in order to lift transverse spaces of depth~$>1$ in an equivariant way with respect to their stabilisers. 
Given a refining sequence $(\xi_1, \dots, \xi_k)$, we endow the space
$$\RR^{k}\times X_{\xi_1, \ldots, \xi_k}$$ 
with the \textbf{refined horoaction} (of depth $k$) of $\Isom(X)_{\xi_1, \ldots, \xi_k}$, which is the product of the canonical action on $X_{\xi_1, \ldots, \xi_k}$ and the translation action on $\RR^k$ provided by the product $\beta_{\xi_1} \times \dots \times \beta_{\xi_k}$ of  the Busemann characters.

\medskip
The case $k=1$ in the following result is covered by Propositions~\ref{prop:lift:facile} and~\ref{prop:ro}. 

\begin{prop}\label{prop:ro_it}
Let $X$ be a proper \cat space and $G < \Isom(X)$ be a closed subgroup acting cocompactly. 
For every refining sequence $(\xi_1, \dots, \xi_k)$,  there is an isometric embedding  $f\colon \RR^{k}\times X_{\xi_1\ldots, \xi_k}\to X$ and a (possibly discontinuous) homomorphism $\ro \colon  G_{\xi_1, \ldots, \xi_k} \to G$ such that $f$  is $\ro$-equivariant for $G_{\xi_1, \ldots, \xi_k}$. 
\end{prop}

\begin{proof}
We shall construct inductively for $0 \leq i \leq k$ an isometric embedding
$$
f^{(i)} \colon \RR^{i} \times X_{\xi_1, \ldots , \xi_{i}}
$$ 
and a homomorphism
$$
\ro^{(i)} \colon G_{\xi_1, \ldots , \xi_{i}} \to G
$$
such that $f^{(i)}$ is $\ro^{(i)}$-equivariant for $ G_{\xi_1, \ldots , \xi_{i}} $. The final step $i=k$ will yield a map $f= f^{(k)} $ and a homomorphism $\ro = \ro^{(k)}$ enjoying the requested properties. 

We start with $f^{(0)} = \Id \colon X \to X$ and $\ro^{(0)} = \Id \colon G \to G$. For the inductive step, we now let $i >0$. 

Viewing $X_{\xi_1, \ldots , \xi_{i-1}}$ as a subspace of $\RR^{i-1} \times X_{\xi_1, \ldots , \xi_{i-1}}$, we may view $\xi_i$ as a point of $\bd (\RR^{i-1} \times X_{\xi_1, \ldots , \xi_{i-1}})$. It therefore makes sense to consider the image $f^{(i-1)}(\xi_i)$, which we denote by $\xi'_i \in \bd X$. In particular we obtain an induced isometric embedding 
$$
\RR^{i-1} \times X_{\xi_1, \ldots , \xi_{i}} \to X_{\xi'_i},
$$
which we also denote by $f^{(i-1)}$. The induction hypothesis ensures that the isometric embedding
$$
\Id \times f^{(i-1)} \colon \RR \times \RR^{i-1} \times X_{\xi_1, \ldots , \xi_{i}} \to \RR \times X_{\xi'_i}
$$
is $\ro^{(i-1)}$-equivariant for $G_{\xi_1, \ldots , \xi_{i}}$ (where the domain $\RR \times \RR^{i-1} \times X_{\xi_1, \ldots , \xi_{i}}$ is endowed with the refined horoaction of depth $i$ of the stabiliser $G_{\xi_1, \ldots , \xi_{i}}$). 

Since $G$ is cocompact, we can choose a sequence of isometries  $\{g^{(i)}_n\}_n$ in $G$ that is radial for $\xi'_i$. Propositions~\ref{prop:lift:facile} and~\ref{prop:ro} then yield an isometric embedding
$$
\tilde f^{(i)} \colon \RR \times X_{\xi'_i} \to X
$$
and a homomorphism
$$
\tilde \ro^{(i)} \colon G_{\xi'_i} \to G
$$
such that $\tilde f^{(i)}$ is $\tilde \ro^{(i)} $-equivariant for $ G_{\xi'_i}$. We then set 
$$
f^{(i)} = \tilde f^{(i)} \circ (\Id \times f^{(i-1)} ) \colon \RR \times \RR^{i-1} \times X_{\xi_1, \ldots , \xi_{i}} \to  X
$$
and 
$$
\ro^{(i)} = \tilde \ro^{(i)} \circ \ro^{(i-1)}.
$$
Since $\Id \times f^{(i-1)}$ was shown to be $\ro^{(i-1)}$-equivariant for $G_{\xi_1, \ldots , \xi_{i}}$, it follows that $f^{(i)}$ is $\ro^{(i)}$-equivariant, which completes the induction step. 
\end{proof}

\begin{remark}\label{rem:ro}
We record for later references that the homomorphism $\ro$ is constructed as a composed map $\ro = \tilde \ro^{(k)} \circ \dots \circ \tilde \ro^{(1)}$, where each $\tilde \ro^{(i)}$ is constructed as a pointwise limit, on its domain of definition,  of a sequence of inner automorphisms of $G$. 
\end{remark}

The latter remark combines with the following elementary fact. 

\begin{lem}\label{lem:iter}
Let $G$ be a second countable locally compact  group, $H \se G$ any subgroup and $\ro_1, \ldots, \ro_k$ a sequence of (possibly discontinuous) group homomorphisms with $\ro_1\colon H\to G$ and  $\ro_{i+1}\colon \mathrm{Im}(\ro_{i})\to G$. 

If each $\ro_i$ is a pointwise limit, on its domain of definition, of a sequence of inner automorphisms of $G$, then so is the composed map $\ro = \ro_k\circ\cdots\circ\ro_1 \colon H \to G$.
\end{lem}

\begin{proof}
It suffices to prove the statement for $k = 2$; the general case follows by induction. Let thus  $F$ be a finite subset of $H$ and $U$ be an open neighbourhood of $\ro_2(H)$ in $G$.  There is $\alpha\in \Inn(G)$ such that $\alpha(\ro_1(F))\se U$. Since $\alpha$ is continuous, there is a neighbourhood $V$ of $\ro_1(F)$ in $G$ such that $\alpha(V)\se U$. Let $\beta\in\Inn(G)$ be such that $\beta(F)\se V$. Then $\alpha\beta(F)\se U$. Notice moreover that the inner automorphisms $\alpha$ and $\beta$ can be chosen in some fixed countable dense subgroup of $G$, which exists since $G$ is assumed second countable. 

Repeating the argument for arbitrarily small neighbourhoods $U$ of $\ro_2(H)$ and arbitrarily large finite subsets $F \se H$, we get a net of inner automorphisms of $G$ of the form $\{\alpha \beta\}_{U, F}$, whose restrictions of $H$ have   the composed homomorphism $\ro = \ro_2 \circ \ro_1$ as a limit point. By construction each element of the net belongs to a common countable subgroup of $G$; therefore the above net is actually a sequence. 
\end{proof}

\begin{cor}\label{cor:ro:extract}
The homomorphism $\ro \colon H \to G$ provided by Proposition~\ref{prop:ro_it} is a  pointwise limit of a sequence  of inner automorphisms of $G$.
\end{cor}

\begin{proof}
Immediate from Remark~\ref{rem:ro} and Lemma~\ref{lem:iter}, recalling that $G$ is second countable because $X$ is proper. 
\end{proof}

We warn the reader that, although $\ro$ is a pointwise limit of a sequence $(\alpha_n)_{n \geq 0}$ of inner automorphisms of $G$,  this does however not imply that $(\alpha_n)_{n \geq 0}$ admits a subsequence whose restriction to $H$ converges pointwise to $\ro$ (this reflects the fact that an uncountable product of metrizable topological spaces need not be first countable). However, there does exist such a subsequence in case the subgroup $H$ is countable.

\bigskip
The \textbf{flat rank} of a \cat space is the supremum of the dimensions of its isometrically embedded Euclidean subspaces. By Theorem~C in~\cite{Kleiner}, the flat rank of a cocompact proper \cat space $X$ is given by~$1+\dim(\bd X)$, where $\bd X$ is endowed with the Tits metric and $\dim$ is the supremum of the topological dimensions of compact subsets. Proposition~\ref{prop:lift:facile} leads to an additional characterisation:

\begin{cor}\label{cor:flat_rank}
Let $X$ be a cocompact proper \cat space. Then the depth of $X$ coincides with its flat rank. 

In fact, we have $\dim(\bd X_{\xi_1, \ldots, \xi_{k}}) \leq \dim(\bd X) -k$ for every refining sequence $(\xi_1, \ldots, \xi_k)$.
\end{cor}

\begin{proof}
The embedding  $\RR^{k}\times X_{\xi_1,\ldots, \xi_k}\to X$ from Proposition~\ref{prop:ro_it} indeed yields the inequality $\dim(\bd X_{\xi_1, \ldots, \xi_{k}}) \leq \dim(\bd X) -k$. Taking a refined point of maximal depth, we deduce in particular that the flat rank of $X$ is at least its depth. On the other hand, it follows readily from the definitions that the flat rank of any \cat space is a lower bound for its depth.
\end{proof}

\section{Compactible subgroups and unimodularity}\label{sec:compactible}
\subsection{Compactible subgroups}
Let $G$ be a locally compact group. We propose the following generalization of the concept of \emph{contractible groups} introduced by P.~M\"uller-R\"omer~\cite{Mueller-Roemer}.

A (not necessarily closed) subgroup $A < G$ is a \textbf{compactible subgroup} with limit $K$ if $K<G$ is a compact subgroup such that for every neighbourhood $U$ of $K$ in $G$ and every finite set $F\se A$ there is $\alpha\in\Aut(G)$ such that $\alpha(F)\se U$. (We write $\Aut(G)$ for the group of continuous automorphisms of $G$.)

A sequence $\{\alpha_n\}$ in $\Aut(G)$ such that $\alpha_n(a)\to K$ for all $a\in A$ will be called a \textbf{compacting sequence}, and we also call a sequence $\{g_n\}$ in $G$ compacting if the sequence of inner automorphisms $\Ad(g_n)$ is compacting.

\medskip
The main difference with the notion of \emph{compaction groups} from~\cite{CCMT} is that we deform $A$ within the ambient group $G$ rather than only within itself.

\begin{prop}\label{prop:CompactibleAmenable}
  If $A < G$ is closed and compactible, then it is  amenable.
\end{prop}

\begin{proof}
By a result of Reiter's (Proposition~1 of~\S3 in~\cite{Reiter71}), it suffices to prove the following: for every finite set $S\se A$ there are non-negative $\fhi\in L^1(G)$ of integral one with $\|s\fhi - \fhi\|_1$ arbitrarily small for all $s\in S$.

Let thus $S$ be given. Let $\psi$ be some left $K$-invariant non-negative $\psi\in L^1(G)$ of integral one.
By continuity of the $G$-representation on $L ^1(G)$, there is a neighbourhood $V$ of the identity in $G$ such that $\|v\psi - \psi\|_1$ is as small as desired for
all $v\in V$. By assumption there is $\alpha\in\Aut(G)$ such that $\alpha(S)\se VK$. Let $\Delta_G(\alpha)$ be the modulus of
$\alpha$ (VII~\S1~No~4 in~\cite{BourbakiINT79_anglais}); then $\fhi=\Delta_G(\alpha)\psi\circ\alpha$ has
the required property.
\end{proof}

The algebraic structure of a compactible subgroup can in fact be described more precisely using the following observation combined with Theorem~\ref{thm:Yamabis}. 

\begin{prop}\label{prop:CompactibleYamabis}
Let $G$ be a locally compact group and $A < G$ be a compactible subgroup. Let $Y = \overline{A \cdot G^\circ}$. 

Then $Y/Y^\circ$ is locally elliptic. 
\end{prop}

\begin{proof}
Since $Y^\circ = G^\circ$, there is no loss of generality in assuming that $G$ is totally disconnected. 

Recall that every compact subgroup of a totally disconnected locally compact group is contained in a compact open subgroup (see Lemma~3.1 from~\cite{CapraceTD}). It follows that every finite set of elements in $A$ is contained in some compact subgroup of $G$. Thus $A$ is locally elliptic. Hence so is its closure $\overline A$ by Lemma~2.1 from~\cite{CapraceTD}.
\end{proof}

\begin{cor}\label{cor:Compactible}
Let $G$ be a locally compact group and $A < G$ be a closed compactible subgroup.

Then $A^\circ \LF(A)$ is open in $A$ and the discrete quotient $A/(A^\circ \LF(A))$ is virtually soluble. 
\end{cor}

\begin{proof}
Let $Y = \overline{A \cdot G^\circ}$. Thus $Y /Y^\circ$ is locally elliptic by Proposition~\ref{prop:CompactibleYamabis}. Let 
$$\pi \colon Y \to Y/\LF(Y)$$
 be the canonical projection. By Theorem~\ref{thm:Yamabis}, the quotient $Y/\LF(Y)$ is a Lie group whose group of components is virtually soluble. Moreover $A$, and hence also $A/A^\circ$, is amenable by Proposition~\ref{prop:CompactibleAmenable}. Therefore $\pi(A)$ is a Lie group and $\pi(A)/\pi(A)^\circ$ is virtually soluble by Proposition~\ref{prop:Lie}. 

Since $A \cap \Ker(\pi)$ is contained in $\LF(A)$, we infer that $A/\LF(A)$ is a quotient of $\pi(A)$, and is thus  itself a Lie group whose group of components is virtually soluble. From Lemma~\ref{lem:quot-Lie}, it follows that $A^\circ \LF(A) /\LF(A) = (A/\LF(A))^\circ$. Therefore $A/A^\circ \LF(A)$ is discrete and virtually soluble, as desired. 
\end{proof}

\subsection{Compactible subgroups  and the refined bordification}\label{sec:com_ref}
Given a refined point $(\xi_1, \ldots, \xi_k; x)$, all the Busemann characters $\beta_{\xi_i}$ are well-defined on the iterated stabiliser $G_{\xi_1, \ldots, \xi_k}$ and hence also on $G_{\xi_1, \ldots, \xi_k;x}$.

\begin{prop}\label{prop:RefinedCompacting}
Let $X$ be a proper \cat space and $G < \Isom(X)$ be a closed subgroup acting cocompactly. For any refined point $(\xi_1, \ldots, \xi_k; x)$, there is a sequence in $G$ which is compacting for the stabiliser $G_{\xi_1, \dots, \xi_k; x} \cap (\bigcap_{i=1}^k \Ker \beta_{\xi_i})$.
\end{prop}

This strengthens a statement for totally disconnected groups contained in Proposition~4.5(iii) in~\cite{CapraceTD}. 

\begin{proof}
The group $G_{\xi_1, \dots, \xi_k; x} \cap (\bigcap_{i=1}^k \Ker \beta_{\xi_i})$ is a point stabiliser for the refined horoaction of $G_{\xi_1, \dots, \xi_k}$ on $\RR^k\times X_{\xi_1, \dots, \xi_k}$. Let $f \colon \RR^k\times X_{\xi_1, \dots, \xi_k} \to X$ and $\ro \colon G_{\xi_1, \dots, \xi_k} \to G$ be the maps provided by Proposition~\ref{prop:ro_it}. The $\ro$-equivariance of $f$ implies that $\ro(G_{\xi_1, \dots, \xi_k; x})$ fixes the point $f(x) \in X$, and is thus relatively compact in $G$. By Corollary~\ref{cor:ro:extract}, the homomorphism $\ro$ is a pointwise limit of some sequence $(\Ad g_n)$ of inner automorphisms of $G$. Since $\ro(G_{\xi_1, \dots, \xi_k; x})$ is contained in the compact group $G_{f(x)}$, it follows that $(\Ad g_n)$ is a compacting sequence for $G_{\xi_1, \dots, \xi_k; x}$.
\end{proof}

\begin{cor}\label{cor:GeomAmenable}
Let $G$ be a locally compact group acting continuously, properly and cocompactly by isometries on $X$. Then the stabiliser of every point in the refined bordification is amenable.
\end{cor}

\begin{proof}
Immediate from Propositions~\ref{prop:CompactibleAmenable} and~\ref{prop:RefinedCompacting}.
\end{proof}

\subsection{Unimodular groups do not fix points at infinity}
We begin with a general restriction for compactible subgroups in unimodular groups.

\begin{prop}\label{prop:unimod:compactible}
Let $G$ be a unimodular locally compact group, $H < G$ a closed subgroup and $\{g_n\}_{n\in\NN}$ a compacting sequence for $H$ in $G$.

If $\{g_n g g_n\inv\}$ remains bounded for each $g\in G$, then $H$ is compact.
\end{prop}

\begin{proof}
Let $K<G$ be a compact subgroup with $g_n h g_n\inv\to K$ for all $h\in H$. Let further $U\se G$ be any non-empty relatively compact open set. We claim that the union of all $g_n U g_n\inv$ is relatively compact.

For this claim there is no loss of generality in assuming $G$ to be $\sigma$-compact. For each compact set $C\se G$, let $U_C=U\cap \bigcap_n g_n\inv C g_n$. As $C$ ranges over a countable exhaustion of $G$, the sets $U_C$ form a cover of $U$ by sets closed in $U$. Therefore, Baire's theorem implies that there is some compact $C$ such that $U_C$ has non-empty interior. Since $U$ has compact closure, we can find a finite set $F\se G$ such that the set $U_C F$ covers $U$. Since the union of all $g_n Fg_n\inv$ is relatively compact, the claim follows.

We thus have a compact set $Q\se G$ containing all $g_n U g_n\inv$. Fix any compact neighbourhood $S$ of $K$ in $G$. Let $p$ be an integer such that $p|U| > |QS|$, where $|\cdot|$ denotes a Haar measure on $G$. If, towards a contradiction, $H$ is non-compact, then we can find $h_1, \ldots, h_p\in H$ such that all $U h_i$ are disjoint. However, when $n$ is large enough, all $g_n U h_i g_n\inv$ lie in $QS$, which is impossible by comparing the measures.
\end{proof}

\begin{proof}[Proof of Theorem~\ref{thm:unimodular}]
Let $\xi \in \bd X$ be a point fixed by the unimodular group $G$. By~\cite[Prop.~7.1]{Caprace-Monod_structure} the set $\Op(\xi)$ is non-empty. Let $\xi' \in \Op(\xi)$. Theorem~\ref{thm:Levi} yields a decomposition
$$G = G_{\xi, \xi'}\cdot G_\xi^\mathrm{u},$$
where $G_\xi^\mathrm{u} = \Ker (\omega_\xi)$ is the kernel of the horoaction. Any radial sequence $\{g_n\}$ for $\xi$ is compacting for $G_\xi^\mathrm{u}$ (this is in fact a very basic case of Proposition~\ref{prop:RefinedCompacting}).  Moreover, $g_n g g_n\inv$ remains bounded for all $g$ in $G=G_\xi$, as observed in Section~\ref{sec:Levi}. Therefore, Proposition~\ref{prop:unimod:compactible} implies that $G_\xi^\mathrm{u}$ is compact. However, a compact normal subgroup of a group acting minimally must be trivial. We deduce that $G = G_{\xi, \xi'}$; in particular, $G$ preserves $P(\xi, \xi')$ which is thus equal to $X$. In conclusion, $X$ admits a product decomposition $X \cong \RR \times X_\xi$ such that the pair $\{\xi, \xi'\}$ corresponds to the boundary of the line-factor. In particular every point $\xi \in \bd X$ fixed by $G$ lies on the boundary of the Euclidean de Rham factor of $X$, as desired.
\end{proof}

\begin{proof}[Proof of Theorem~\ref{thm:lattice:NoFixed}]
Since $\Isom(X)$ admits a lattice, it is unimodular. Therefore Theorem~\ref{thm:unimodular} implies that $\Isom(X)$ does not have a fixed point in $\bd X$. Now the ``geometric Borel density'' established in Theorem~1.1 of~\cite{Caprace-Monod_discrete} yields the conclusion.
\end{proof}

\section{Amenable isometry groups}\label{sec:Amen}

\begin{flushright}
\begin{minipage}[t]{0.75\linewidth}\itshape\small
La m\'ethode de cet artificieux Bouillon, c'est [\ldots] de rester, en fin de compte, le moyenneur des situations.
\begin{flushright}
\upshape Charles-Augustin Sainte-Beuve, \emph{Causeries du lundi}, tome 12, 1856
\end{flushright}
\end{minipage}
\end{flushright}

\subsection{Refined flats}

We define a  \textbf{refined flat of depth $k$} in $X$ to be a flat contained in a space of the form $X_{\xi_1, \dots, \xi_k}$, where $(\xi_1, \dots, \xi_k)$ is a refining sequence. Thus a flat of depth~$0$ is a flat in $X$. 

A key ingredient needed in the proof of Theorem~\ref{thm:StructureAmenable} is the theorem by Adams--Ballmann from~\cite{AB98}. We shall use the following formulation of their result.

\begin{prop}[Adams--Ballmann]\label{prop:AB}
Let $X$ be a proper \cat space of finite depth (e.g.\ cocompact, see Proposition~\ref{prop:TransverseSpace}(\ref{it:FiniteDepth})). 

Every amenable group acting continuously by isometries on $X$ stabilises a {refined flat}. 
\end{prop}
 
\begin{proof}
Follows from a repeated application of the Adams--Ballmann theorem~\cite{AB98}; the hypothesis that the depth is finite guarantees that the process terminates after finitely many steps.
\end{proof}

Before proceeding to the proof of Theorem~\ref{thm:StructureAmenable}, we shall undertake to prove Theorem~\ref{thm:converseAB}, providing a converse to Proposition~\ref{prop:AB} and thus establishing Corollary~\ref{cor:converseAB}. We point out that Theorem~\ref{thm:converseAB} is not a direct consequence of the amenability of the kernel of the refined horoaction, because the image of a (refined) flat stabiliser in the isometry group of that flat need not be closed, see Remark~\ref{rem:disc}. Therefore, we first need to record a consequence of a theorem by Kazhdan--Margulis (see Theorem~8.16 from~\cite{Raghunathan}).

\medskip
To this end, we recall that the space of all closed subgroups of a locally compact group $G$ is compact when endowed with the Chabauty topology~\cite{Chabauty}. Moreover the Chabauty space of closed subgroups is metrisable provided $G$ is second countable; this happens e.g. when $G$ is a closed subgroup of $\Isom(X)$, where $X$ is a proper metric space. 

\begin{prop}\label{prop:MargulisLemma}
Let $G$ be a locally compact group whose identity component $G^\circ$ has no non-trivial compact normal subgroup. Let $(\Gamma_j)_{j\in J}$ be a net of discrete free subgroups converging in the Chabauty topology to a closed subgroup $L \se G$. 

Then the neutral component $L^\circ$ is an abelian Lie group. 
\end{prop}

\begin{proof}
By van Dantzig's theorem~\cite[page~18]{vanDantzig31}, the group $G$ has an open subgroup $O$ (thus containing $G^\circ$) such that $O/G^\circ$ is compact. By Yamabe's theorem $\LF(O)$ is compact and $G^\circ$ is a Lie group. Moreover Theorem~\ref{thm:Yamabis} implies that  $\LF(O) G^\circ \cong \LF(O) \times G^\circ$ is open in $O$. Upon replacing $O$ by that subgroup, we shall assume henceforth that $O$ is a direct product of a connected Lie group and a compact (profinite) group. 

Since $O$ is open, we have $L \cap O = \lim_j (\Gamma_j \cap O)$. Since the desired conclusion concerns the group $L^\circ \se G^\circ \se O$,  there is therefore no loss of generality in assume that all the subgroups $\Gamma_j$ are contained in $O$ or, equivalently, that $G=O$. Then $G \cong K \times G^\circ$, where $K = \LF(G)$ is compact and the identity component $G^\circ$ is a Lie group. Since $G^\circ$ is a Lie group, so is   $L^\circ$. 

Let $\pi \colon G \to G^\circ$ be the canonical projection, which is proper since $K$ is compact. Thus $\pi(\Gamma_j)$ is a discrete free group for all $j$. Let $U$ be an open relatively compact identity neighbourhood of the connected Lie group $G^\circ$ as provided by the Kazhdan--Margulis theorem (see Theorem~8.16 from~\cite{Raghunathan}). Since $L^\circ \se G^\circ$, we have 
$$
L^\circ = \pi(L^\circ) =  \la U \cap \pi(L^\circ) \ra \se \la U \cap \pi(L) \ra \se \lim_j \la U \cap \pi(\Gamma_j) \ra.
$$
The Kazhdan--Margulis Theorem ensures that  $\la U \cap \pi(\Gamma_j) \ra$ is nilpotent, hence cyclic, for all $j$. Since a limit of abelian subgroups is abelian, it follows that $L^\circ$ is abelian. 
\end{proof}

We shall also need an elementary general fact related to Lemma~\ref{lem:quot-Lie} above:

\begin{lem}\label{lem:image_conn}
Let $N$ be a closed normal subgroup of a locally compact group $H$. Then the image of $H^\circ$ in $H/N$ is dense in $(H/N)^\circ$. In particular, if $N$ is compact, then we obtain an identification $H^\circ/(H^\circ\cap N) \cong (H/N)^\circ$.
\end{lem}

This statement fails for general topological groups; recall for instance that $\RR$ is the quotient of a totally disconnected group (see Exercice~17 for III\,\S2 in~\cite{BourbakiTGIII}).

\begin{proof}[Proof of the lemma]
The image of $H^\circ$ in $H/N$ lies in $(H/N)^\circ$; denote its closure by $L$. Now $(H/N)^\circ/L$ is a connected group contained in $(H/N)/L$. Writing the latter as $H/\overline{H^\circ N}$, we see that it is a quotient of the totally disconnected group $H/H^\circ$. For \emph{locally compact} groups, total disconnectedness passes to quotients (by van Dantzig's theorem~\cite[page~18]{vanDantzig31}). Thus $(H/N)^\circ/L$ is also totally disconnected, hence trivial as claimed. The additional statement for $N$ compact follows from the properness of the map $H\to H/N$.
\end{proof}




Here is a last preparation for the proof of Theorem~\ref{thm:converseAB}.

\begin{lem}\label{lem:refined_flat}
Let $G$ be a group acting isometrically on a complete \cat space $X$ and let $Y\se X$ be a $G$-invariant closed convex subspace of finite coradius in $X$. If $G$ preserves a refined flat of $X$, then it also preserves a refined flat (of the same dimension) of $Y$.
\end{lem}

\begin{proof}
The statement is formulated in such a way that it passes to transverse spaces, since $\bd X= \bd Y$ and since there is a canonical isometric map of finite coradius $Y_\xi\to X_\xi$ for all $\xi\in \bd X$. Therefore, it suffices to deal with the case of depth zero, i.e.\ a $G$-invariant flat $F\se X$. Since the projection $\pi\colon X\to Y$ is $G$-equivariant, it is enough to show that $\pi$ is isometric on $F$. The function $x\mapsto d(x,Y)$ is convex and bounded, therefore it is constant on $F$. Denote this constant by $D$. We claim that $\pi(F)$ is convex; this claim will finish the proof because the function $y\mapsto d(y, F)$ is constant on $\pi(F)$ and therefore the sandwich lemma~\cite[II.2.12]{Bridson-Haefliger} will apply.

Let thus $x,x'\in F$ and let $y$ be any point on the segment $[\pi(x), \pi(x')]$. By convexity of the metric, the distance between $y$ and the corresponding point $z$ on the segment $[x,x']$ (for simultaneous affine parametrisations) is at most $D$. However, $y$ is in $Y$ and $z$ is in $F$. By uniqueness of the projection, we deduce $y=\pi(z)$ and the claim is proved.
\end{proof}

\begin{proof}[Proof of Theorem~\ref{thm:converseAB}]
Suppose for a contradiction that $H < G:=\Isom(X)$ is a non-amenable closed subgroup which stabilises a refined flat $F \se X_{\xi_1, \dots, \xi_k}$. Upon replacing $X$ by a minimal non-empty $G$-invariant closed convex subset, we may  assume that $G$ acts minimally. Indeed, Lemma~\ref{lem:refined_flat} ensures that $H$ still preserves a refined flat (still denoted by $F$ as above) and the restriction homomorphism to the isometries of the minimal subset has compact kernel. This reduction ensures that $G$ has no non-trivial compact normal subgroup, and hence the unique maximal compact normal subgroup $\LF(G^\circ)$ of $G^\circ$ is trivial.

\medskip

Notice that $H \cap \big(\bigcap_{i=1}^k \Ker(\beta_{\xi_i})\big)$ is closed and co-amenable in $H$. Therefore it is non-amenable, and upon replacing $H$ by that intersection, we may assume that $H$ annihilates each Busemann character $\beta_{\xi_i}$. 

Let $N$ be the pointwise stabiliser of $F$ in $H$. Thus $N$ is amenable by Corollary~\ref{cor:GeomAmenable}, so that the quotient $H/N$ endowed with the quotient topology is non-amenable. The $H$-action on $F$ induces a continuous embedding $\pi \colon H/N \to \Isom(F)$. Since $\Isom(F)$ is a Lie group, so is $H/N$ by Proposition~\ref{prop:Lie}(i). Moreover, since $\Isom(F)$ is amenable, the neutral component $(H/N)^\circ$ must be amenable as well, since non-amenable connected Lie groups cannot be continuously embedded in amenable Lie groups (this follows e.g.\ from the fact that any non-trivial continuous homomorphism of a simple Lie group is proper, see~\cite[Lemma~5.3]{BM96} or~\cite{Cornulier_lengths}). Therefore the group of components $(H/N) / (H/N)^\circ$ is non-amenable, and hence contains a non-abelian free subgroup by Proposition~\ref{prop:Lie}(ii). 

Denoting by $H^1$ the preimage of $(H/N)^\circ$ in $H$, we infer that $H^1$ is an open normal amenable subgroup of $H$, and that the quotient  $H/H^1$ contains a non-abelian free subgroup. Let now $a, b \in H$ be elements which freely generate a non-abelian free subgroup modulo $H^1$. Then $\la a, b \ra \cap H^1$ is trivial, since otherwise there would be a word in the generators $a, b$ which gets killed in the quotient $H/H^1$, contradicting the presupposed freeness.  It follows that $\Gamma = \la a, b \ra$ is a discrete non-abelian free subgroup of $H$, hence of $G$, which stabilises the refined flat $F \se X_{\xi_1, \dots, \xi_k}$, acts faithfully on $F$, and  annihilates each Busemann character $\beta_{\xi_i}$.

\medskip
Let $f \colon \RR^k \times X_{\xi_1, \dots, \xi_k} \to X$ and $\ro \colon \Gamma \to G$ be the maps provided by Proposition~\ref{prop:ro_it}. Thus $F' = f(\RR^k \times F)$ is a flat in $X$. Moreover,  the equivariance afforded by Proposition~\ref{prop:ro_it} ensures that the map $\ro$ is injective and that its image $\ro(\Gamma)$ stabilises $F'$.

By assumption the image of $\Gamma$ in the Lie group $\Isom(F)$ is non-discrete. In particular the closure of its image has a non-trivial identity component. Consequently the closure of the image of $\ro(\Gamma)$ in $\Isom(F')$ has a non-trivial identity component. 

Let $K$ be the pointwise stabiliser of $F'$ in $G$. Thus $K$ is compact, and it follows from Lemma~\ref{lem:image_conn} that 
$\overline{\ro(\Gamma)}^\circ$ has a non-trivial image in $\Isom(F')$. Moreover Proposition~\ref{prop:Lie}(i) implies that $\overline{\ro(\Gamma)}^\circ$ is open in $\overline{\ro(\Gamma)}$. 

\medskip
By Corollary~\ref{cor:ro:extract}, the homomorphism $\ro \colon \Gamma \to G$ is a pointwise limit of a sequence $(\alpha_n)_{n \geq 0}$ of inner automorphisms of $G$. Since $\Gamma$ is countable, we may assume, upon extracting, that the sequence of restrictions $(\alpha_n |_\Gamma)_n$ converges pointwise to $\ro$. Up to a further extraction, we may also assume that the sequence of conjugates $\alpha_n(\Gamma)$ converges in the Chabauty topology to some closed subgroup $L$. It then follows that $\ro(\Gamma)$, and hence also the closure $\overline{\ro(\Gamma)}$, is contained in $ L$. 

By  Proposition~\ref{prop:MargulisLemma} the identity component $L^\circ$ is an abelian Lie group. Thus $\overline{\ro(\Gamma)}^\circ$ is an abelian open normal subgroup of $\overline{\ro(\Gamma)}$. Therefore  $\ro(\Gamma) \cap \overline{\ro(\Gamma)}^\circ$ is an abelian normal subgroup of $\ro(\Gamma)$ which is dense in $\overline{\ro(\Gamma)}^\circ$.  By construction $\ro(\Gamma)$ is a non-abelian free group, and has thus no non-trivial abelian normal subgroup. Consequently $ \ro(\Gamma) \cap \overline{\ro(\Gamma)}^\circ$ is trivial, and so is $\overline{\ro(\Gamma)}^\circ$. This contradicts the fact that $\overline{\ro(\Gamma)}^\circ$ has a non-trivial image in $\Isom(F')$. 
\end{proof}

\subsection{Structure of amenable closed subgroups}\label{sec:structure}

We finally record the following classical characterisation of amenability for connected locally compact groups:

\begin{prop}\label{prop:iwasa:C}
A connected locally compact group is amenable if and only if it is \{connected soluble\}-by-compact.
\end{prop}

\begin{proof}
Let $G$ be a connected locally compact group. By the solution of Hilbert's fifth problem~\cite[4.6]{Montgomery-Zippin}, $G$ is compact-by-Lie. Furstenberg has shown that a connected Lie group is amenable if and only if it is \{connected soluble\}-by-compact~\cite[Theorem~1.7]{Furstenberg}. Finally, Iwasawa has shown that the class of  \{connected soluble\}-by-compact l.c.\ groups is closed under extensions~\cite[Theorem~18]{Iwasawa49}.
\end{proof}

\begin{proof}[Proof of Theorem~\ref{thm:StructureAmenable}]
The `if' part is clear. Conversely, let $H < \Isom(X)$ be a closed amenable subgroup. Then $H^\circ$ is soluble-by-compact in view of Proposition~\ref{prop:iwasa:C} and hence~(1) holds.

By  Proposition~\ref{prop:AB}, we find an $H$-invariant refined flat  $F \subset X_{\xi_1, \dots, \xi_k}$ of depth~$k$. In particular we have a continuous homomorphism
$$\varphi \colon H \to \Isom(F) \cong \RR^n \rtimes O(n),$$
where $n = \dim F$. We endow the image $\varphi(H)$ with the quotient topology from $H$. By Proposition~\ref{prop:Lie}, the amenable locally compact group $\varphi(H)$ is a Lie group and its group of components $\varphi(H)/\varphi(H)^\circ$ is virtually soluble. By Lemma~\ref{lem:quot-Lie}, we have $\varphi(H^\circ) = \varphi(H)^\circ$. Thus  the normal subgroup $K = H^\circ \cdot \Ker(\varphi)$ is open in $H$ and the quotient $H/K \cong \varphi(H)/\varphi(H)^\circ$ is virtually soluble.

The kernel $A = \Ker(\varphi)$ fixes $F$ pointwise, and is thus compactible by Proposition~\ref{prop:RefinedCompacting}. Therefore $A^\circ \LF(A)$ is open in $A$ and the quotient $A/A^\circ \LF(A)$ is virtually soluble by Proposition~\ref{prop:CompactibleYamabis}. In particular the group
$$H^\circ \cdot A^\circ \LF(A) = H^\circ  \LF(A) $$
is open in $K = H^\circ \cdot A$, hence in $H$. Since $\LF(A)$ is characteristic in $A$, we have $\LF(A) \leq \LF(H)$, and we infer that $H^\circ  \LF(H)$ is open in $H$, which proves assertion~(2).

For (3), recall from the above that  $H/K$ and $A/A^\circ \LF(A) $ are both virtually soluble. Therefore, so is $ H^\circ A/H^\circ \LF(A) = K/ H^\circ  \LF(A) $. Virtual solubility being stable under group extensions, we infer that  $H/H^\circ  \LF(A)$, and~(3) follows since $\LF(A) \leq \LF(H)$. 
\end{proof}

The following lemma is well-known; a proof can be found in~\cite[Proposition~4]{HRV}. 

\begin{lem}\label{lem:residually:finite}
A homomorphic image of a finitely generated group in any compact group is residually finite (thus the same holds for its image in a locally elliptic group).\qed
\end{lem}

\begin{proof}[Proof of Corollary~\ref{cor:SimpleGps}]
Let $X$ be a cocompact proper \cat space with an isometric $\Gamma$-action and let $H$ be the closure of the image of $\Gamma$ in $\Isom(X)$. The image of $\Gamma$ in $H/ (H^\circ \LF(H))$ is trivial since this quotient is virtually soluble by Theorem~\ref{thm:StructureAmenable}; thus $H=H^\circ \LF(H)$.  Since $H^\circ$ is soluble-by-compact by Theorem~\ref{thm:StructureAmenable}, we deduce from Lemma~\ref{lem:residually:finite} that the image of $\Gamma$ in $H/ \LF(H)$ is soluble and hence trivial. It follows $H=\LF(H)$, but this implies that $H$ is compact since $\Gamma$ is finitely generated; we finally conclude that $H$ is trivial by applying again the residual finiteness argument.
\end{proof}

We now turn to Corollary~\ref{cor:T} and refer to~\cite{Cannon-Floyd-Parry} for an introduction to Thompson's groups $F$, $T$ and $V$. Since $T$ and $V$ are simple and contain $F$, it suffices to produce a non-trivial kernel in $F$ to ensure that the entire $T$-\ or $V$-action be trivial. Thus the following statement implies Corollary~\ref{cor:T}.

\begin{cor}
Any isometric action of Thompson's group $F$ on any proper cocompact \cat space factors through the abelianisation $F\to F/[F,F]\cong \ZZ^2$.
\end{cor}

\begin{proof}
Let $X$ be any proper cocompact \cat space and assume that $F$ acts isometrically on $X$; let $\pi\colon F\to\Isom(X)$ be the associated group homomorphism. According to Corollary~2.3 of~\cite{Caprace-Monod_discrete}, the group $F$ fixes a point in the compactification $\overline{X'}$ whenever it acts isometrically on any proper \cat space $X'$.
Thus in particular it satisfies the conclusion of the Adams--Ballmann theorem and therefore we can deduce as in Proposition~\ref{prop:AB} that $F$ preserves a refined flat (in fact a refined point) for $X$. Now Theorem~\ref{thm:converseAB} implies that the closure $H<\Isom(X)$ of $\pi(F)$ is amenable; we can therefore apply Theorem~\ref{thm:StructureAmenable}. Recalling that the derived subgroup $[F,F]$ is simple, we deduce that $\pi([F,F])$ lies in $H^\circ \LF(H)$. Recall next that $[F,F]$ contains many isomorphic copies of $F$; we chose one such subgroup $F_1\cong F$ in $[F,F]$. By Lemma~\ref{lem:residually:finite}, the homomorphism from $F_1$ to $\big(H^\circ \LF(H)\big)/H^\circ$ must annihilate $[F_1, F_1]$ since the latter is simple. Thus $\pi([F_1, F_1])$ is contained in the soluble-by-compact group $H^\circ$. Choosing once again an isomorphic copy $F_2\cong F$ in $[F_1,F_1]$, one more application of Lemma~\ref{lem:residually:finite} shows that $\pi([F_2, F_2])$ lies i
 n the soluble radical of $H^\circ$. This shows that $\pi([F_2, F_2])$ is trivial since $[F_2, F_2]$ is simple (non-abelian). We have established that $\pi$ has a non-trivial kernel; this implies that $\pi$ is trivial on all of $[F,F]$ since that group is simple.
\end{proof}

The strategy of the above proof is the following: (i)~prove a refined fixed point (or flat) property for a subgroup (not necessarily involving amenability); (ii)~apply Theorem~\ref{thm:converseAB} to deduce that its closure in $\Isom(X)$ is amenable; (iii)~play off the structure provided by Theorem~\ref{thm:StructureAmenable} against the structure of the given group.

\smallskip
This strategy can be implemented in many other cases; here is a first example:

\begin{cor}\label{cor:GGG}
Let $G$ be a finitely generated group isomorphic to $G\times G$. For any isometric $G$-action  on any proper cocompact \cat space, there is a decomposition $G\cong G\times G$ for which that action factors through the first factor.

\medskip
Moreover, if one chooses a priori decompositions $G\cong G^n$ for $n\in \NN$, then for any isometric $G$-action on any proper cocompact \cat space of flat rank~$<n$, one of the $n$ factors in $G^n$ acts trivially.
\end{cor}

We recall that every countable group can be embedded in a group $G\cong G\times G$ as above~\cite[Cor.~6]{Meier82}. Moreover, there are torsion-free examples; for instance, following the proof of Proposition~7 in~\cite{Meier82}, consider two copies $U,V$ of the Baumslag--Solitar group $BS(2,3)$:
$$U=\big\langle a,t : t\inv a^2 t = a^3 \big\rangle, \kern10mm V=\big\langle b,s : s\inv b^2 s = b^3 \big\rangle.$$
Then $t$ and $[a, t\inv a t]$ freely generate a free group, as do $[b, s\inv b s]$ and $s$. Define $T$ to be the free product of $U$ and $V$ amalgamated over these free subgroups (with the generators identified in the order given above). Then Meier exhibits a (non-trivial) finitely generated subgroup $G\cong G\times G$ in the full product $T^\NN$. Notice that $T$ is obtained by successive HNN-extensions and amalgamated products, starting with infinite cyclic groups and amalgamating over free subgroups (of rank one and two). Elementary Bass-Serre theory (or classical HNN-theory~\cite[\S5]{Neumann54}) thus implies that $T$ is torsion-free; hence $G$ is torsion-free as well.

\begin{proof}[Proof of Corollary~\ref{cor:GGG}]
It suffices to prove the second statement since the factors of $G^n$ can be suitably re-arranged as $G\times G$. Therefore, we fix some decomposition $G\cong G^n$ and an isometric $G$-action on a proper cocompact \cat space $X$ of flat rank~$< n$. In order to be able to pass to successive transverse spaces, we only retain the weaker information that $X$ is proper of flat rank~$< n$ (which is preserved thanks to Proposition~\ref{prop:ro_it}). Such a space cannot contain a product of $n$ unbounded convex subspaces. Therefore, the splitting theorem of~\cite{Monod_superrigid} (specifically, Corollary~10 therein) implies that either $G$ fixes a point at infinity of $X$, or one of the factors of $G^n$ preserves a bounded subset of $X$, hence fixing a point.  In any case, it follows by induction that one factor $H\cong G$ of $G^n$ fixes a refined point. Applying successively Theorem~\ref{thm:converseAB} and Theorem~\ref{thm:StructureAmenable}, we find that the image of $H$ is $\Isom
 (X)$ is obtained by various extensions of soluble and residually finite groups (appealing to Lemma~\ref{lem:residually:finite} for the residual finiteness). However, a finitely generated group isomorphic to its own square does not have any non-trivial finite or abelian quotient; thus $H$ acts trivially indeed.
\end{proof}

It was pointed out in the introduction, as a consequence of Theorem~\ref{thm:DiscreteAmen}, that the wreath product $\ZZ \wr \ZZ$ cannot be a discrete subgroup of $\Isom(X)$ for any proper cocompact \cat space $X$. The assumption of discreteness is essential in that observation, since the group $\ZZ \wr \ZZ$ does admit faithful actions by automorphisms on regular locally finite trees. This should however be contrasted with the torison-free elementarily amenable group $G=  \Big((\ZZ \wr \ZZ) \wr \ZZ\Big) \wr \ZZ \cdots$, which is  more precisely defined as the increasing union of the groups
$$G_1=\ZZ, \kern10mm G_{n+1}= G_n\wr \ZZ = \Big(\bigoplus_{\ZZ} G_n\Big)\rtimes \ZZ,$$
where $G_n$ is identified with its copy at the coordinate zero in $G_{n+1}$.

\begin{cor}\label{cor:couronne}
For any isometric $G$-action  on any proper cocompact \cat space, there is a subgroup of $G$ isomorphic to $G$ which acts trivially.
\end{cor}

\begin{proof}
This follows from Theorem~\ref{thm:StructureAmenable} once one observes that the kernel of any homomorphism from $G$ to any soluble group and to any locally \{residually finite\} group contains a subgroup isomorphic to $G$.
\end{proof}


\subsection{Discrete amenable subgroups and soluble groups of finite rank}
In view of Theorem~\ref{thm:StructureAmenable}, discrete amenable subgroups of $\Isom(X)$ are \{locally finite\}-by-\{virtually soluble\}.  The main result of this section is the much more precise statement of Theorem~\ref{thm:DiscreteAmen}. We recall that a group is said be of \textbf{Pr\"ufer rank $r$} if every finitely generated subgroup can be generated by at most $r$ elements, and if $r$ is the smallest such integer.  This notion is sometimes simply called the \textbf{rank} in the literature and we shall do so for brevity's sake. For example, the additive group $\QQ^n$ is of rank $n$. Soluble groups of finite rank have been studied extensively since Mal'cev's  foundational work in the early 1950's.   An overview of their theory, including many far-reaching results can be consulted in the book~\cite{LennoxRobinson}. At this point, we only record the elementary fact that the class of groups of finite rank is closed under group extensions: more precisely, if $G$ is a group 
 with a normal subgroup $N$ such that $N$ and $G/N$ have rank $r$ and $r'$ respectively, then $G$ has rank at most $r+r'$. 

\medskip
The proof of Theorem~\ref{thm:DiscreteAmen} has a structure roughly parallel to that of Theorem~\ref{thm:StructureAmenable}, and will be given at the end of the section. We shall need to study stabilisers of refined flats;  the discreteness assumption on $\Gamma$ is exploited in combination with  some rigidity properties on discrete amenable subgroups of Lie groups. 

\medskip
Our first task is thus to collect those various subsidiary results, notably on Lie groups. Although they  are probably well-known, we nevertheless supply proofs,  by lack of appropriate references. 
The first is a variation on Bieberbach's theorem. 

\begin{prop}\label{prop:Bieberbach}
Let $n \geq 0$. Given a closed subgroup $H $ of the Lie group $\RR^n \rtimes O(n)$, the group of components $H/H^\circ$ is \{free abelian of rank~$\leq n$\}-by-finite.
\end{prop}

Notice that the order of the finite quotient cannot be bounded by a function of $n$, since the circle group $S^1$ contains cyclic subgroups of arbitrarily large order.

\begin{proof}[Proof of Proposition~\ref{prop:Bieberbach}]
Let $\EE^n$ be the Euclidean $n$-space, so that $\Isom(\EE^n) = \RR^n \rtimes O(n)$. We work by induction on $n$, the base case $n=0$ being trivial. 

If $H$ stabilises a proper Euclidean subspace, let $N$ be the (compact) kernel of the $H$-action on that space; thus $H/N$ can be identified with a closed subgroup of $\RR^m \rtimes O(m)$ for $m<n$. By induction, $(H/N)/(H/N)^\circ$ is \{free abelian of rank~$\leq m$\}-by-finite. But this group is a quotient of $H/H^\circ$ by Lemma~\ref{lem:image_conn}. The kernel of this quotient is finite, so that in this case we are done since a finite-by-\{free abelian of rank~$\leq m$\}-by-finite group is \{free abelian of rank~$\leq m$\}-by-finite.

We assume henceforth that $H$ acts minimally on $\EE^n$.  Let $F \subset E_n$ be a minimal $H^\circ$-invariant Euclidean subspace, and set $m = \dim(F)$. The union of the collection of all such subspaces is itself a Euclidean subspace which is  $H$-invariant, and must therefore coincide with $\EE^n$ by minimality. Therefore we obtain an $H$-invariant product decomposition
$$\EE^n  \cong \EE^m \times \EE^{n-m} $$
with the additional property that the $H^\circ$-action on the second factor $\EE^{n-m}$ is trivial. Let $p \colon H \to \Isom(\EE^m)$ and $p' \colon H \to \Isom(\EE^{n-m})$ be the canonical projections. We have seen that $H^\circ \leq  \Ker(p') = H'$.  

Pick a base point $x \in \EE^m$ and set $H_x = \Stab_H(x)$. Thus the intersection $H_x \cap H'$ is compact,  and hence the projection $p'(H_x) \cong H_x / (H_x \cap H')$ is a closed subgroup of  $\Isom(\EE^{n-m})$. 

By~\cite{DiScala}, the neutral component $H^\circ$ acts transitively on its minimal subspace $F \cong \EE^m$. Equivalently $p(H^\circ)$ is transitive on $\EE^m$, and thus $H = H_x \cdot H^\circ$. In particular  $H = H_x \cdot H'$, hence we have $p'(H) \cong H/H' \cong H_x / (H_x \cap H')\cong p'(H_x)$.  Therefore, the projection $p'(H) $  is a closed subgroup of $\Isom(\EE^{n-m})$, which must be discrete since $p'(H^\circ)$ is trivial. In particular $H/H' $  is virtually \{free abelian of rank~$\leq n-m$\} by a version of Bieberbach's theorem, see e.g.\ Theorem~2 in~\cite{Oliver}.

As $H^\circ \leq H'$, the transitivity of  $p(H^\circ)$  on $\EE^m$ also implies $H' = (H_x \cap H') \cdot H^\circ$. It follows that $H'/H^\circ$ is compact, hence finite.  Consequently $H/H^\circ$ is an extension of $H/H'$ by a finite normal subgroup. Thus $H/H^\circ$ is virtually a free abelian group of rank~$\leq n-m$, as desired. 
\end{proof}


\begin{prop}\label{prop:Auslander}
Let $G$ be a connected Lie group. Then there is $r = r(G)$ such that every discrete amenable subgroup  $\Gamma < G$ has the following properties:
\begin{enumerate}[(i)]
\item $\Gamma$ is \{torsion-free polycyclic of rank~$\leq r$\}-by-finite. 

\item $\LF(\Gamma)$ is finite and $\Gamma/\LF(\Gamma)$ is \{torsion-free polycyclic of rank~$\leq r$\}-by-\{finite of order~$\leq r$\}.
\end{enumerate}
\end{prop}

\begin{proof}
Let $\Gamma < G$ be a discrete amenable subgroup. Then $\Gamma$ is virtually soluble by Proposition~\ref{prop:Lie}. By Corollary~8.5 from~\cite{Raghunathan}, any non-empty collection of soluble subgroups of $G$ has a maximal element. Recalling that polycyclic groups are precisely those soluble groups satisfying the ascending chain condition on subgroups (see e.g.~\cite[1.3.1]{LennoxRobinson}), we infer that $\Gamma$ is virtually polycyclic. Consequently it is finitely generated and  virtually \{torsion-free polycyclic\}. The bound on the rank of the polycylic kernel in terms of $G$ follows from~\cite{Auslander1960} (and can be traced back to Zassenhaus~\cite{Zassenhaus}). Thus (i) holds, and the finiteness of $\LF(\Gamma)$  follows.

Let now $\Gamma < G$ be a discrete subgroup which is \{torsion-free polycyclic of rank~$n$\}-by-finite. In order to establish (ii), we proceed by induction on $n$, the case $n=0$ being trivial. Assume now $n>0$. Then $\Gamma$ has a non-trivial free abelian normal subgroup, say $\ZZ^m \cong A  < \Gamma$, with $m \leq n$. The conjugation action of $\Gamma$ on $A$ yields a homomorphism $\pi : \Gamma \to \GL_m(\ZZ)$ whose kernel is $\centra_\Gamma(A) \cdot A$. 

The index of a torsion-free normal subgroup $\pi(\Gamma)$ is bounded above by a constant depending only on $m$. Therefore $\pi(\Gamma)$ is \{torsion-free polycyclic of rank~$\leq r'$\}-by-\{finite of order~$\leq r'$\} for some $r'$ depending only on $n$.

On the other hand, the induction hypothesis applies to the quotient $\Ker(\pi)/A$, which is a central extension of $\Ker(\pi)$. Bearing in mind that a central extension of a locally elliptic group is \{locally elliptic\}-by-\{torsion-free abelian\} (see Proposition~\ref{prop:GenUshakov}), we infer that the group $\Ker(\pi)/\LF(\Ker (\pi))$  is \{torsion-free polycyclic of rank~$\leq r''$\}-by-\{finite of order~$\leq r''$\} for some $r''$ depending only on $n$. Assertion~(ii) follows since $\LF(\Ker (\pi)) < \LF(\Gamma)$. 
\end{proof}

Finally, we record an abstract group theoretic property.

\begin{lem}\label{lem:AbstractRank}
Let $\Gamma$ be a countable group. Assume there are constants $r, M \geq 0$ such that for every finitely generated subgroup $\Lambda < \Gamma$, the radical $\LF(\Lambda)$ is finite and the quotient $\Lambda/\LF(\Lambda)$ is  \{torsion-free polycyclic of rank~$\leq r$\}-by-\{finite of order~$\leq M$\}. 

Then $\Gamma/\LF(\Gamma)$ is \{torsion-free soluble of rank~$\leq r$\}-by-finite.
\end{lem}

\begin{proof}
We enumerate the elements of $\Gamma$ as $\Gamma = \{\gamma_1, \gamma_2, \dots \}$. For each $n >0$, let $\Lambda_n = \la \gamma_1, \dots, \gamma_n \ra $ and $R_n = \LF(\Lambda_n)$. 

For $m \geq n$, the finite  quotient $R_n / R_m \cap R_n$ embeds in $\Lambda_m / R_m$. Its order is thus bounded above by $M$.  Since $R_n$ is finite, we find a sequence of indices $(m_i)_i$ such that $m_{i+1} > m_i \geq n$ and that the intersection $R_{m_i} \cap R_n$ is constant and has index at most~$M$ in $R_n$. Setting $R'_n = R_{m_i} \cap R_n$ and noticing that $\bigcap_i R_{m_i}$ is  finite group that is normalised by all elements of  $\Gamma$, we infer that 
$R'_n \se \bigcap_i R_{m_i} \se \LF(\Gamma)$. This shows that $[R_n  : R_n \cap \LF(\Gamma) ] \leq M$ for all $n$. It follows that the image of $\Lambda_n$ in the quotient $\Gamma/ \LF(\Gamma)$ is  \{finite of order~$\leq M$\}-by-\{torsion-free polycyclic of rank~$\leq r$\}-by-\{finite of order~$\leq M$\}.

We infer that there is some $N$ such that every finitely generated subgroup of the quotient $H = \Gamma/ \LF(\Gamma)$ is  \{torsion-free polycyclic of rank~$\leq r$\}-by-\{finite of order~$\leq N$\}. In particular $H$  is of finite rank. 
It only remains to show that $H$ is virtually torsion-free.

To this end, let $H'$ be the subgroup of $H$ generated by all $N!$ powers. Thus $H'$ is normal and the quotient $H/H'$ is of finite exponent~$\leq N!$. In particular, for any finitely generated subgroup $\Lambda < H$, the image of $\Lambda $ in $H/H'$ is polycyclic-by-finite of finite exponent, hence finite. Consequently $H/H'$ is locally finite. On the other hand  $H/H'$ is of finite rank since $H$ is so. Remark that a locally finite group of finite rank and of finite exponent must be finite by the solution to the restricted Burnside problem\footnote{It is quite possible that a softer argument could be provided in the present setting by invoking Corollary~2 page~141 of~\cite{Shunkov}.} (see~\cite{Zelmanov_Burnside} and references therein).  We conclude that  $H/H'$ is finite. 

Finally, given $h \in H'$, there is a finitely generated subgroup $\Lambda$ of $H$ such that $h $ belongs to the subgroup $\Lambda'$ of $\Lambda$ generated by all $N!$ powers of elements of $\Lambda$. Since $\Lambda$ is \{torsion-free\}-by-\{finite of order~$\leq N\}$, it follows that $\Lambda'$ is torsion-free. Consequently  $H'$ is torsion-free as well. 
\end{proof}

\begin{proof}[Proof of Theorem~\ref{thm:DiscreteAmen}]
Let $G=\Isom(X)$ and $\Gamma < \Isom(X)$ be a discrete amenable subgroup. We already know from Theorem~\ref{thm:StructureAmenable} that $\Gamma/\LF(\Gamma)$ is virtually soluble. We need to show that it is in fact virtually \{torsion-free soluble of rank~$\leq r$\} for some $r$ depending only on $X$. 

\medskip
We start with a preliminary observation.

Pick a basepoint $x \in X$ and let $R >0$ be large enough so that every $G$-orbit meets the closed ball $B(x, R)$ of radius $R$ around $x$. 
 For each $z \in X$, the stabiliser $G_z$ is compact. Its image in the totally disconnected quotient $G/G^\circ$ is compact as well, and hence contained in some compact open subgroup of $G/G^\circ$ by~\cite[Lemma~3.1]{CapraceTD}. Let $O_z < G$ be the preimage in $G$ of such a compact open subgroup. Thus $O_z$ is open, contains $G_z\cdot G^\circ$, and the quotient $O_z/G^\circ$ is compact.   

We next remark that the union $\bigcup_{z \in B(x, R)} G_z$  is a relatively compact subset of $G$. We can therefore find a finite set $z_1, \dots, z_n$ such that $\bigcup_{z \in B(x, R)} G_z \subset \bigcup_{i=1}^n O_{z_i}$. Let $O = \bigcap_{i=1}^n O_{z_i}$. Thus $O$ is open, hence contains $G^\circ$, and the quotient $O/G^\circ$ is compact. 

We claim that there is a uniform bound $N$ such that for all $z \in  B(x, R)$, we have $[ G_z G^\circ  : G_z G^\circ \cap O] \leq N$. The fact that this index is finite is clear since $O$ is open. Since $G_zG^\circ \subset \bigcup_{i=1}^n O_{z_i}$, it follows that $ G_z G^\circ $ is covered by left cosets of $O$ in the various groups $O_{z_i}$. Therefore we obtain
$$[G_z G^\circ  : G_z G^\circ \cap O]  \leq \sum_{i=1}^n \;  {[O_{z_i} : O]} .$$
The claim follows by setting $N =  \sum_{i=1}^n \, { [O_{z_i} : O] }$.

\medskip
We now return to the discrete amenable subgroup $\Gamma < G$. 
By Proposition~\ref{prop:AB}, there is a $\Gamma$-invariant refined flat $F \subset X_{\xi_1, \dots, \xi_k}$ of depth~$k \geq 0$. 

Let $f\colon \RR^{k}\times X_{\xi_1\ldots, \xi_k}\to X$ and $\ro' \colon  G_{\xi_1, \ldots, \xi_k} \to G$ be the maps provided by Proposition~\ref{prop:ro_it}. By assumption $\Gamma$ stabilises the flat $\RR^k \times F \se  \RR^{k}\times X_{\xi_1\ldots, \xi_k}$. Since $f$ is $\ro'$-equivariant, it follows that $\ro'(\Gamma)$ stabilises the flat $F' = f(\RR^k \times F) \se X$.  

\medskip
By cocompactness, there is some $h \in G$ such that $h(F')$ meets the ball $B(x, R)$. We set $\ro = \Ad h \circ \ro'$, so that $\ro(\Gamma)$ stabilises $h(F')$. 

Let $H = \overline{\ro(\Gamma) }$. Denote by $K < H$ the pointwise stabiliser of $h(F')$, which is thus compact and normal in $H$. The quotient $H/K$   embeds as a closed subgroup of the Lie group $\Isom(h(F'))$. Since $X$ is cocompact, the dimension of $F'$ is bounded above by a constant depending only on $X$.  By Lemma~\ref{lem:quot-Lie} (or Lemma~\ref{lem:image_conn}), the neutral component $H^\circ$ maps onto the neutral component of $H/K$. Consequently  Proposition~\ref{prop:Bieberbach} ensures that  the quotient $H/(H^\circ K)$ is \{free abelian of rank~$\leq r'$\}-by-finite, where $r'=r'(X)$ depends only on $X$. Let $\Gamma' = \ro\inv( H^\circ K )$. Thus $\Gamma/\Gamma'$ is \{free abelian of rank~$\leq r'$\}-by-finite.

By construction $K$ fixes some point $z \in B(x, R)$. Therefore we have $\ro(\Gamma') \leq H^\circ K \leq G^\circ G_z$. We finally set $\Gamma'' = \ro\inv(O)$. Thus $[\Gamma' : \Gamma''] \leq [G^\circ G_z : G^\circ G_z \cap O] \leq N$ be the claim above.

Let $\Lambda < \Gamma''$ be a finitely generated subgroup. Thus $\ro(\Lambda) < O$. Since $\Lambda$ is finitely generated and $O$ is open, it follows that the conjugate $\Lambda_n = hg_n \Lambda g_n\inv h\inv$ is contained in $O$ for all sufficiently large $n$. 

By Theorem~\ref{thm:Yamabis}, the radical $\LF(O)$ is compact and the group $O/\LF(O)$ is an almost connected Lie group. Moreover, the number of its connected components depends only on $X$. 
Notice moreover that $\Lambda_n \cap \LF(O)$ is a  finite normal subgroup of $\Lambda_n$,  thus contained in $\LF(\Lambda_n)$. 
By applying Proposition~\ref{prop:Auslander} to the quotient $\Lambda_n / \Lambda_n \cap \LF(O)$, we deduce  that   $\LF(\Lambda_n)$ is finite and that $\Lambda_n/\LF(\Lambda_n)$ is \{torsion-free polycyclic of rank~$\leq r''$\}-by-\{finite of order~$\leq r''$\}, where  $r''$ depends only on $X$. The same properties therefore hold for $\Lambda$, since it is conjugate to $\Lambda_n$. 

Recalling that $[\Gamma' : \Gamma''] \leq N$, we infer that for every finitely generated subgroup $\Lambda \leq \Gamma'$, the radical $\LF(\Lambda)$ is finitely generated and the quotient $\Lambda / \LF(\Lambda)$ is  \{torsion-free polycyclic of rank~$\leq r''$\}-by-\{finite of order~$\leq N r''$\}. We are thus in a position to apply Lemma~\ref{lem:AbstractRank} to $\Gamma'$, which ensures that $\Gamma'$ is \{torsion-free soluble of rank~$\leq r''$\}-by-finite.

Recall further that $\Gamma/\Gamma'$ is  \{free abelian of rank~$\leq r'$\}-by-finite. Therefore $\Gamma$ is \{torsion-free soluble of rank~$\leq r$\}-by-finite, where $r = r' + r''$ depends only on $X$. This concludes the proof. 
\end{proof}

\subsection{An easy comment on nilpotent groups}
For the record, we indicate a variation on the Adams--Ballmann theorem~\cite{AB98} that is apparently not available in the literature, giving a stronger conclusion in the special case of nilpotent groups. This easy comment is not needed for the results of this paper.

\begin{prop}\label{prop:AB:nilpotent}
Let $G$ be a nilpotent group acting by isometries on a proper \cat space $X$. Then either the $G$ preserves a flat or it fixes a point $\xi \in \bd X$ and annihilates the Busemann character $\beta_\xi$.
\end{prop}

This will be proved inductively, using the following relative statement for general groups.

\begin{prop}\label{prop:AB:centre}
Let $G$ be a group acting by isometries on a proper \cat space $X$. Then either the centre $\centra(G)$ preserves a flat or $G$ fixes a point $\xi \in \bd X$ and annihilates the Busemann character $\beta_\xi$.
\end{prop}

\begin{proof}[Proof of Proposition~\ref{prop:AB:centre}]
For any finite subset $F\se \centra(G)$ and any $\epsilon>0$ we define
$$X_{F, \epsilon} = \big\{x\in X : d(z x, x) \leq |z| + \epsilon \ \forall\,z\in F\big\},$$
where $|z|$ denotes the infimal displacement length of $z$. By centrality, this closed convex subset of $X$ is non-empty and $G$-invariant. Letting $F$ grow and $\epsilon$ shrink, we obtain a net of subsets, monotone decreasing with respect to inclusion. Assume first that the intersection $Y$ of all $X_{F, \epsilon}$ is non-empty. Since all elements of $\centra(G)$ have constant displacement on $Y$, there is a decomposition $Y\cong \RR^n \times Y_0$ (with $n\geq 0$) such that every element of $\centra(G)$ acts by (possibly trivial) translations on $\RR^n$ and trivially on $Y_0$, see Theorem~II.6.15 in~\cite{Bridson-Haefliger}. Thus $\centra(G)$ preserves indeed a flat.

We assume henceforth that the net $X_{F, \epsilon}$ has empty intersection. This implies that the distance from a base-point $p$ to $X_{F, \epsilon}$ tends to infinity. Denote by $p_{F, \epsilon}$ the projection of $p$ to $X_{F, \epsilon}$. Any accumulation point $\xi\in \bd X$ of the net $p_{F, \epsilon}$ will correspond to a $G$-invariant Busemann function since $X_{F, \epsilon}$ is $G$-invariant, finishing the proof. (Compare Proposition~2.1(2) in~\cite{AB98}.)
\end{proof}

\begin{proof}[Proof of Proposition~\ref{prop:AB:nilpotent}]
We argue by induction on the nilpotency class of $G$; the base case is when $G$ is the trivial group. For the inductive step, we apply Proposition~\ref{prop:AB:centre} and need only consider the case where $\centra(G)$ preserves a flat $F$ in $X$. Upon passing to a subflat of minimal dimension, we see that $F$ is minimal as (non-empty) convex $\centra(G)$-invariant subset. The union $U\se X$ of all such flats splits canonically as $U\cong F\times C$ for some proper \cat space $C$ endowed with a canonical $G/\centra(G)$-action, see Theorem~4.3 in~\cite{Caprace-Monod_structure}. We now apply the induction hypothesis to this $G/\centra(G)$-action on $C$. If $G/\centra(G)$ preserves a flat $E\se C$, then $G$ preserves the flat $F\times E \se X$ and we are done. If on the other hand $G/\centra(G)$ fixes a point $\xi \in \bd C \se \bd X$ and annihilates the corresponding Busemann character, then we are also done since $\centra(G)$ acts trivially on $C$.
\end{proof}


\providecommand{\bysame}{\leavevmode\hbox to3em{\hrulefill}\thinspace}
\providecommand{\MR}{\relax\ifhmode\unskip\space\fi MR }
\providecommand{\MRhref}[2]{%
  \href{http://www.ams.org/mathscinet-getitem?mr=#1}{#2}
}
\providecommand{\href}[2]{#2}

\end{document}